\documentclass[a4paper, 12pt]{article}

\usepackage[T1]{fontenc}
\usepackage[utf8]{inputenc}

\usepackage{times}
\usepackage{amsfonts, amssymb, amsthm, graphicx, mathtools, epstopdf}
\usepackage{accents}

\usepackage[colorlinks, citecolor=blue, urlcolor=blue]{hyperref}

\usepackage{graphicx, subfigure}
\usepackage{microtype}
\usepackage{setspace}
\usepackage[margin=1in]{geometry}

\newtheorem{definition}{Definition}
\newtheorem{theorem}{Theorem}

\newtheorem{lemma}{Lemma}

\newcommand{\Sscr}{{\mathcal{S}}}
\newcommand{\Tscr}{{\mathcal{T}}}
\newcommand{\Fscr}{{\mathcal{F}}}
\newcommand{\Gscr}{{\mathcal{G}}}
\newcommand{\Bscr}{{\mathcal{B}}}
\newcommand{\Hscr}{{\mathcal{H}}}

\newcommand{\Xscr}{{\mathcal{X}}}
\newcommand{\Yscr}{{\mathcal{Y}}}

\newcommand{\Dscr}{{\mathcal{D}}}

\newcommand{\E}{{\mathbb{E}}}
\newcommand{\Rscr}{{\mathcal{R}}}
\newcommand{\ip}[1]{\left\langle #1 \right\rangle}

\title{Co-clustering of Nonsmooth Graphons}
\author{David Choi}
\begin{document}
\maketitle

\begin{abstract}
	
Performance bounds are given for exploratory co-clustering/ blockmodeling of bipartite graph data, where we assume the rows and columns of the data matrix are samples from an arbitrary population. This is equivalent to assuming that the data is generated from a nonsmooth graphon. It is shown that co-clusters found by any method can be extended to the row and column populations, or equivalently that the estimated blockmodel approximates a blocked version of the generative graphon, with estimation error bounded by $O_P(n^{-1/2})$. Analogous performance bounds are also given for degree-corrected blockmodels and random dot product graphs, with error rates depending on the dimensionality of the latent variable space.

\end{abstract}

\section{Introduction}

In the statistical analysis of network data, blockmodeling (or community detection) and its variants are a popular class of methods that have been tried in many applications, such as modeling of communication patterns \cite{blondel2008fast}, academic citations \cite{ji2014coauthorship}, protein networks \cite{airoldi2009mixed}, online behavior \cite{latouche2011overlapping,  traud2011comparing}, and ecological networks \cite{girvan2002community}. 

In order to develop a theoretical understanding, many recent papers have established consistency properties for the blockmodel. In these papers, the observed network is assumed to be generated using a set of latent variables that assign the vertices into groups (the `` communities''), and the inferential goal is to recover the correct group membership from the observed data. Various conditions have been established under which recovery is possible \cite{bickel2009nonparametric, bickel2011method} and also computationally tractable \cite{cai2015robust, chen2012fitting, krzakala2013spectral, newman2013spectral, sussman2012consistent}. Additionally, conditions are also known under which no algorithm can correctly recover the group memberships \cite{decelle2011asymptotic, mossel2013proof}. 

The existence of a true group membership is central to these results. In particular, they assume a generative model in which all members of the same group are statistically identical. This implies that the group memberships explain the entirety of the network structure.
In practice, we might not expect this assumption to even approximately hold, and the objective of finding ``true communities'' could be difficult to define precisely, so that a more reasonable goal might be to discover group labels which partially explain structure that is evident in the data. Comparatively little work has been done to understand blockmodeling from this viewpoint. 

To address this gap, we consider the problem of blockmodeling under model misspecification. We assume that the data is generated not by a blockmodel, but by a much larger nonparametric class known as a graphon. This is equivalent to assuming that the vertices are sampled from an underlying population, in which no two members are identical and the notion of a true community partition need not exist. In this setting, blockmodeling might be better understood not as a generative model, but rather as an exploratory method for finding high-level structure: by dividing the vertices into groups, we divide the network into subgraphs that can exhibit varying levels of connectivity. This is analogous to the usage of histograms to find high and low density regions in a nonparametric distribution. Just as a histogram replicates the binned version of its underlying distribution without restrictive assumptions, we will show that the blockmodel replicates structure in the underlying population when the observed network is generated from an arbitrary graphon. 

Our results are restricted to the case of bipartite graph data. Such data arises naturally in many settings, such as customer-product networks where connections may represent purchases, reviews, or some other interaction between people and products. 

The organization of the paper is as follows. Related work is discussed in Section \ref{sec: related work}. In Section \ref{sec: co-clustering}, we define the blockmodeling problem for bipartite data generated from a graphon, and present a result showing that the blockmodel can detect structure in the underlying population. In Section \ref{sec: other models}, we discuss extensions of the blockmodel, such as mixed-membership, and give a result regarding the behavior of the excess risk in such models. Section \ref{sec: proofs} contains a sketch and proof for the main theorem. Auxilliary results and extensions are proven in the Appendix.

\section{Related Works} \label{sec: related work}

The papers \cite{airoldi2013stochastic, gao2014rate, klopp2015oracle, olhede2014network}, and \cite{choi2014co} are most similar to the present work, in that they consider the problem of approximating a graphon by a blockmodel. The papers \cite{airoldi2013stochastic, gao2014rate, klopp2015oracle} and \cite{olhede2014network} consider both bipartite and non-bipartite graph data, and require the generative graphon to satisfy a smoothness condition, with \cite{gao2014rate} establishing a minimax error rate and \cite{klopp2015oracle} extending the results to a class of sparse graphon models. In a similar vein, \cite{sussman2012universally} shows consistent and computationally efficient estimation assuming a type of low rank generative model. While smoothness and rank assumptions are natural for many non-parametric regression problems, it seems difficult to judge whether they are appropriate for network data and if they are indeed necessary for good performance. 

In \cite{choi2014co} and in this present paper, which consider only bipartite graphs, the emphasis is on exploratory analysis. Hence no assumptions are placed on the generative graphon. Unlike the works which assume smoothness or low rank structure, the object of inference is not the generative model itself, but rather a blocked version of it (this is defined precisely in Section \ref{sec: co-clustering}). This is reminiscent of some results for confidence intervals in nonparametric regression, in which the interval is centered not on the generative function or density itself, but rather on a smoothed or histogram-ed version \cite[Sec 5.7 and Thm 6.20]{wasserman2006all}. The present paper can be viewed as a substantial improvement over \cite{choi2014co}; for example, Theorem \ref{th: co-blockmodel} improves the rates of convergence from $O_P(n^{-1/4})$ to $O_P(n^{-1/2})$, and also applies to computationally efficient estimates. 




\section{Co-clustering of nonsmooth graphons} \label{sec: co-clustering}

In this section, we give a formulation for co-clustering (or co-blockmodeling) in which the rows and columns of the data matrix are samples from row and column populations, and correspond to the vertices of a bipartite graph. We then present an approximation result which implies that any co-clustering of the rows and columns of the data matrix can be extended to the populations. Roughly speaking, this means that if a co-clustering ``reveals structure'' in the data matrix, then similar structure will also exist at the population level.

\subsection{Problem Formulation}

\paragraph{Data generating process} Let $A \in \{0,1\}^{m \times n}$ denote a binary $m \times n$ matrix representing the observed data. For example, $A_{ij}$ could denote whether person $i$ rated movie $j$ favorably, or whether gene $i$ was expressed under condition $j$. 

We assume that $A$ is generated by the following model, in which each row and column of $A$ is associated with a latent variable that is sampled from a population: 

\begin{definition}[Bipartite Graphon] \label{def: graphon}
Given $m$ and $n$, let $x_1,\ldots,x_m$ and $y_1,\ldots,y_n$ denote i.i.d. uniform $[0,1]$ latent variables 
\begin{align*}
x_1,\ldots,x_m  \stackrel{iid}{\sim} \operatorname{Unif}[0,1] \qquad \textrm{and} \qquad
y_1,\ldots,y_n  \stackrel{iid}{\sim} \operatorname{Unif}[0,1].
\end{align*}
Let $\omega:[0,1]^2 \mapsto [0,1]$ specify the distribution of $A \in \{0,1\}^{m \times n}$, conditioned the latent variables $\{x_i\}_{i=1}^m$ and $\{y_j\}_{j=1}^n$, 
\begin{align*}
A_{ij} & \sim \operatorname{Bernoulli}\left(\omega(x_i,y_j)\right), \qquad i \in [m], j \in [n]
\end{align*}
where the Bernoulli random variables are independent. 
\end{definition}
We will require that $\omega$ be measurable and bounded between $0$ and $1$, but may otherwise be arbitrarily non-smooth. We will use $\Xscr = [0,1]$ and $\Yscr = [0,1]$ to denote the populations from which $\{x_i\}$ and $\{y_j\}$ are sampled.

\paragraph{Co-clustering} In co-clustering, the rows and columns of a data matrix $A$ are simultaneously clustered to reveal submatrices of $A$ that have similar values. When $A$ is binary valued, this is also called blockmodeling (or co-blockmodeling). 

Our notation for co-clustering is the following. Let $K$ denote the number of clusters. Let $S \in [K]^m$ denote a vector  identifying the cluster labels corresponding to the $m$ rows of $A$, e.g., $S_i = k$ means that the $i$th row is assigned to cluster $k$. Similarly, let $T \in [K]^n$ identify the cluster labels corresponding to the $n$ rows of $A$. Given $(S,T)$, let $\Phi_A(S,T) \in [0,1]^{K \times K}$ denote the normalized sums for the submatrices of $A$ induced by $S$ and $T$:
\[ \left[ \Phi_A(S,T)\right]_{st} = \frac{1}{mn} \sum_{i=1}^m \sum_{j=1}^n A_{ij} 1(S_i = s, T_j = t), \qquad s,t \in [K].\]
Let $\pi_S \in [0,1]^K$ and $\pi_T \in [0,1]^K$ denote the fraction of rows or columns in each cluster:
\begin{equation*}
\pi_S(s) = \frac{1}{m} \sum_{i=1}^m 1(S_i = s) \qquad \text{and} \qquad \pi_T(t) = \frac{1}{n} \sum_{j=1}^n 1(T_j=t).
\end{equation*}
Let the average value of the $(s,t)$th submatrix be denoted by $\hat{\theta}_{st}$, given by 
\[ \hat{\theta}_{st} = \frac{[\Phi_A(S,T)]_{st}}{\pi_S(s) \pi_T(t)}.\]
Generally, $S$ and $T$ are chosen heuristically to make the entries of $\hat{\theta}$ far from the overall average of $A$. A common approach is to perform k-means clustering of the spectral coordinates for each row and column of $A$ \cite{rohe2012co}. Heterogeneous values of $\hat{\theta}$ can be interpreted as revealing subgroups of the rows and columns in $A$.

\paragraph{Population co-blockmodel} Given a co-clustering $(S,T)$ of the rows and columns of $A$, we will consider whether similar subgroups also exist in the unobserved populations $\Xscr$ and $\Yscr$. Let $\sigma: \Xscr \mapsto [K]$ and $\tau: \Yscr \mapsto [K]$ denote mappings that co-cluster the row and column populations $\Xscr$ and $\Yscr$. Let $\Phi_\omega(\sigma,\tau) \in [0,1]^{K \times K}$ denote the integral of $\omega$ within the induced co-clusters, or the blocked version of $\omega$:
\[ \left[ \Phi_\omega(\sigma,\tau)\right]_{st} = \int_{\Xscr \times \Yscr} \omega(x,y)\,1(\sigma(x) = s, \tau(y) = t)\, dx\, dy, \qquad s,t \in [K]. \]
Let $\Phi_\omega(S,\tau) \in [0,1]^{K \times K}$ denote the integral of $\omega$ within the induced co-clusters, over $\{x_1,\ldots,x_n\} \times \Yscr$:
\[ \left[ \Phi_\omega(S, \tau) \right]_{st} = \frac{1}{m} \sum_{i=1}^m \int_\Yscr \omega(x_i, y)\,1(S_i = s, \tau(y) = t)\, dy. \]
Let $\pi(\sigma)$ and $\pi(\tau)$ denote the fraction of the population in each cluster:
\begin{equation*} 
\pi_{\sigma}(s)= \int_\Xscr 1(\sigma(x)=s)\, dx \qquad \text{and} \qquad \pi_\tau(t) = \int_\Yscr 1(\tau(y) = t)\, dy.
\end{equation*}
Theorem \ref{th: co-blockmodel} will show that for each clustering $S,T$, there exists $\sigma:\Xscr \mapsto [K]$ and $\tau:\Yscr \mapsto [K]$ which cluster  the populations $\Xscr$ and $\Yscr$ such that $\Phi_A(S,T) \approx \Phi_\omega(S,\tau)$ and $\Phi_A(S,T) \approx \Phi_\omega(\sigma,\tau)$, as well as $\pi_S \approx \pi_{\sigma}$ and $\pi_T \approx \pi_{\tau}$, implying that subgroups found by co-clustering $A$ are indicative of similar structure in the populations $\Xscr$ and $\Yscr$.

\subsection{Approximation Result for Co-clustering}

Theorem \ref{th: co-blockmodel} states that for each $(S,T) \in [K]^m \times [K]^n$, there exists population co-clusters $\sigma_S:\Xscr\mapsto [K]$ and $\tau_{\,T}:\Yscr\mapsto [K]$ such that $\Phi_A(S,T) \approx \Phi_\omega(S,\tau_T) \approx \Phi_\omega(\sigma_S,\tau_T)$, and also $\pi_S \approx\pi_{\sigma_S}$ and $\pi_T \approx \pi_{\tau_{\,T}}$.

\begin{theorem}\label{th: co-blockmodel}
Let $A \in \{0,1\}^{m \times n}$ be generated by some $\omega$ according to Definition \ref{def: graphon}, with fixed ratio $m/n$. Let $(S,T)$ denote vectors in $[K]^m$ and $[K]^n$ respectively, with $K \leq n^{1/2}$.
\begin{enumerate}
\item For each $T \in [K]^n$, there exists $\tau_{\,T}:\Yscr \mapsto [K]$ such that
\begin{equation} \label{eq: phi_tau}
\max_{S,T \in [K]^m \times [K]^n} \| \Phi_A(S,T) - \Phi_\omega(S,\tau_{\,T})\| + \|\pi_T - \pi_{\tau_{\,T}}\| = O_P\left(\sqrt{\frac{ K^2 \log n}{n}}\right) 
\end{equation}
\item For each $S \in [K]^m$, there exists $\sigma_S:\Xscr \mapsto [K]$, such that
\begin{equation} \label{eq: phi_sigma}
\sup_{S,\tau \in [K]^m \times [K]^\Yscr} \| \Phi_\omega(S,\tau) - \Phi_\omega(\sigma_S,\tau)\| + \|\pi_S -\pi_{\sigma_S}\|  = O_P\left(\sqrt{\frac{ K^2 \log m}{m}}\right) 
\end{equation}
\item Combining \eqref{eq: phi_tau} and \eqref{eq: phi_sigma} yields
\begin{align} \label{eq: phi_both}
\max_{S,T \in [K]^m \times [K]^n} \|\Phi_\omega(\sigma_S, \tau_{\,T}) - \Phi_A(S,T)\| + \|\pi_T - \pi_{\tau_{\,T}}\| + \|\pi_S -\pi_{\sigma_S}\| \\
\nonumber \qquad {} = O_P\left(\sqrt{\frac{ K^2 \log n}{n}}\right).
\end{align}
\end{enumerate}
\end{theorem}

\paragraph{Remarks for Theorem \ref{th: co-blockmodel}} 

To give context to Theorem \ref{th: co-blockmodel}, suppose that $A \in \{0,1\}^{m \times n}$ represents product-customer interactions, where $A_{ij}=1$ indicates that product $i$ was purchased (or viewed, reviewed, etc.) by customer $j$. We assume $A$ is generated by Definition \ref{def: graphon}, meaning that the products and customers are samples from populations. This could be literally true if $A$ is sampled from a larger data set, or the populations might only be conceptual, perhaps representing future products and potential customers.

Suppose that we have discovered cluster labels $S \in [K]^m$ and $T \in [K]^n$ producing a density matrix $\hat{\theta}$ with heterogeneous values. These clusters can be interpreted as product categories and customer subgroups, with heterogeneity in $\hat{\theta}$ indicating that each customer subgroup may prefer certain product categories over others. We are interested in the following question: will this pattern generalize to the populations $\Xscr$ and $\Yscr$? Or is it descriptive, holding only for the particular customers and products that are in the data matrix $A$? 

An answer is given by Theorem \ref{th: co-blockmodel}. Specifically, \eqref{eq: phi_tau} and \eqref{eq: phi_both} show different senses in which the co-clustering $(S,T)$ may generalize to the underlying populations. \eqref{eq: phi_tau} implies that the customer population $\Yscr$ will be similar to the $n$ observed customers in the data, regarding their purchases of the $m$ observed products when aggregated by product category. \eqref{eq: phi_both} implies a similar result, but for their purchases of the entire population $\Xscr$ of products aggregated by product category, as opposed only to the $m$ observed products in the data. 

Since Theorem \ref{th: co-blockmodel} holds for all $(S,T)$, it applies regardless of the algorithm that is used to choose the co-blockmodel. It also applies to nested or hierarchical clusters. If \eqref{eq: phi_tau} or \eqref{eq: phi_both} holds at the lowest level of hierarchy with $K$ classes, then it also holds for the aggregated values at higher levels as well, albeit with the error term increased by a factor which is at most $K$.

Theorem \ref{th: co-blockmodel} controls the behavior of $\Phi_A, \pi_{S},$ and $\pi_{T}$, instead of the density matrix $\hat{\theta}$ which may be of interest. However, since $\hat{\theta}$ is derived from the previous quantities, it follows that Theorem \ref{th: co-blockmodel} also implies control of $\hat{\theta}$ for all co-clusters involving $\gg m^{1/2}$ rows or $\gg n^{1/2}$ columns.

All constants hidden by the $O_P(\cdot)$ notation in Theorem \ref{th: co-blockmodel} are universal, in that they do not depend on $\omega$ (but do depend on the ratio $m/n$). 

\section{Application of Theorem \ref{th: co-blockmodel} to Bipartite Graph Models} \label{sec: other models}

In many existing models for bipartite graphs, the rows and columns of the adjacency matrix $A \in \{0,1\}^{m\times n}$ are associated with latent variables that are not in $\Xscr$ and $\Yscr$, but in other spaces $\Sscr$ and $\Tscr$ instead. In this section, we give examples of such models and discuss their estimation by minimizing empirical squared error. We define the population risk as the difference between the estimated and actual models, under a transformation mapping $\Xscr$ to $\Sscr$ and $\Yscr$ to $\Tscr$. Theorem \ref{th: generalized} shows that the empirical error surface converges uniformly to the population risk. The theorem does not assume a correctly specified model, but rather that the data is generated by an arbitrary $\omega$ following Definition \ref{def: graphon}.

\subsection{Examples of Bipartite Graph Models} 

We consider models in which the rows and columns of $A$ are associated with latent variables that take values in spaces other than $\Xscr$ and $\Yscr$. To describe these models, we will use $S = (S_1,\ldots, S_m)$ and $T=(T_1,\ldots,T_n)$ to denote the row and column latent variables, and $\Sscr$ and $\Tscr$ to denote their allowable values. Let $\Theta$ denote a parameter space. Given $\theta \in \Theta$, let $\omega_\theta: \Sscr \times \Tscr \mapsto [0,1]$ determine the distribution of $A$ conditioned on $(S,T)$, so that the entries $\{A_{ij}\}$ are conditionally independent Bernoulli variables, with $\mathbb{P}(A_{ij}=1|S,T) = \omega_\theta(S_i,T_j)$. 

\begin{enumerate}
\item {\bf Stochastic co-blockmodel with $K$ classes:} Let $\Sscr = \Tscr = [K]$ and $\Theta = [0,1]^{K \times K}$. For $\theta \in  \Theta$, let $\omega_\theta$ be given by
\[\omega_\theta(s,t) = \theta_{s t}, \qquad s,t \in \Sscr \times \Tscr\]
where $s \in \Sscr$ and $t \in \Tscr$ are row and column co-cluster labels.
\item {\bf Degree-corrected co-blockmodel \cite{karrer2011stochastic, zhao2012consistency}:} Let $\Sscr = \Tscr = [K] \times [0,1)$ and $\Theta = [0,1]^{K \times K}$. Given $u,v \in [K]$ and $b,d \in [0,1)$, let $s = (u,b)$ and $t=(v,d)$. Let $\omega_\theta$ be given by
\[\omega_\theta(s,t) = b d\, \theta_{uv}, \qquad s,t \in \Sscr \times \Tscr.\]
In this model, $u,v \in [K]$ are co-cluster labels, and $b,d \in \lbrack 0,1)$ are degree parameters, allowing for degree heterogeneity within co-clusters. 
\item {\bf Random Dot Product \cite{hoff2002latent, sussman2012universally}:} Let $\Sscr = \Tscr = \{c \in [0,1)^d: \|c\| \leq 1\}$. Let $\omega$ be given by
\[\omega(s, t) = s^T t, \qquad s,t \in \Sscr \times \Tscr.\]

\item {\bf Dot Product + Blockmodel:} Models 1-3 are instances of a somewhat more general model. Let $\Dscr = \{c \in [0,1)^d: \|c\|\leq 1\}$. Let $\Sscr = \Tscr = [K] \times \Dscr$ and $\Theta = [0,1]^{K \times K}$. Given $u, v \in [K]$ and $b,d \in \Dscr$, let $s=(u,b)$ and $t=(v,d)$. Let $\omega_\theta$ be given by
\begin{equation}\label{eq: model}
\omega_\theta(s,t) = b^Td\, \theta_{uv}.
\end{equation}

\end{enumerate}

\subsection{Empirical and Population Risk} \label{sec: risk}


Given a data matrix $A \in \{0,1\}^{m \times n}$, and a model specification $(\Sscr, \Tscr, \Theta)$, one method for estimating $(S,T,\theta) \in \Sscr^m \times \Tscr^n \times \Theta$ is to minimize the empirical squared error $R_A$, given by
\begin{align*} 
R_A(S,T;\theta) = \frac{1}{nm} \sum_{i=1}^m \sum_{j=1}^n \left(A_{ij} - \omega_\theta(S_i,T_j)\right)^2.
\end{align*}
Generally, the global minimum of $R_A$ will be intractable to compute, so a local minimum is used for the estimate instead.


If a model $(S,T,\theta)$ is found by minimizing or exploring the empirical risk surface $R_A$, does it approximate the generative $\omega$? We will define the population risk in two different ways:
 \begin{enumerate}
\item {\bf Approximation of $\omega$ by $\omega_\theta$:} Let $\sigma$ and $\tau$ denote mappings $\Xscr \mapsto \Sscr$ and $\Yscr \mapsto \Tscr$, and let $R_\omega$ be given by
\begin{align*}
	R_\omega(\sigma, \tau; \theta) = \int_{\Xscr \times \Yscr} \left[\omega(x,y) - \omega_\theta(\sigma(x), \tau(y))\right]^2 dx dy,
\end{align*}
denoting the error between the mapping $(x,y) \mapsto \omega_\theta(\sigma(x), \tau(y), \theta)$ and the generative $\omega$. If there exists $\theta$ such that $R_\omega(\sigma,\tau;\theta)$ is low for some $\sigma:\Xscr \mapsto \Sscr$ and $\tau:\Yscr \mapsto \Tscr$, then $\omega_\theta$ (or more precisely, its transformation $(x,y) \mapsto \omega_\theta(\sigma(x), \tau(y))$ can be considered a good approximation to $\omega$.

\item {\bf Approximation of $\sigma^* = \arg \min_\sigma R_\omega(\sigma, \tau, \theta)$ by $S$:} Overloading notation, let $R_\omega(S,\tau,;\theta)$ denote
\begin{align*}
	R_\omega(S, \tau; \theta) = \frac{1}{m} \sum_{i=1}^m \int_{\Yscr} \left[\omega(x_i, y) - \omega_\theta(S_i, \tau(y))\right]^2 dy. 
\end{align*}
To motivate this quantity, consider that given $(\tau,\theta)$, the optimal partition $\sigma^*:[0,1] \mapsto [K]$ is the greedy assignment for each $x \in [0,1]$:
\[ \sigma^*(x) = \arg \min_{s \in [K]} \int_{ 0,1 } \left[\omega(x_i, y) - \omega_\theta(s, \tau(y))\right]^2 dy. \]
If there exists $(S,\theta)$ such that $R_\omega(S,\tau; \theta)$ is low for some choice of $\tau$, then $S$ can be considered a good approximation to the corresponding $\{\sigma^*(x_i)\}_{i=1}^m.$
\end{enumerate}
Theorem \ref{th: generalized} will imply that for models of the form \eqref{eq: model}, minimizing $R_A$ is asymptotically a reasonable proxy for minimizing $R_\omega$ (by both metrics described above), with rates of convergence depending on the covering numbers of $\Sscr$ and $\Tscr$.

\subsection{Convergence of the Empirical Risk Function} \label{sec: th2}

Theorem \ref{th: generalized} gives uniform bounds between $R_A$ and $R_\omega$ for models of form \eqref{eq: model}. Specifically, for each choice of $(S,T) \in \Sscr^m \times \Tscr^n$, there exists transformations $\sigma_{S}:\Xscr \mapsto \Sscr$ and $\tau_{T}:\Yscr \mapsto \Tscr$ such that $R_A(S,T;\theta) \approx R_\omega(\sigma_{S},\tau_{T};\theta) \approx R_\omega(S,\tau_{T};\theta)$, up to an additive constant and with uniform convergence rates depending on $d$ and $K$. As a result, minimization of $R_A(S,T;\theta)$ is a reasonable proxy for minimizing $R_\omega$, by either measure defined in Section \ref{sec: risk}.

In addition, the mappings $\sigma_S$ and $\tau_T$ will resemble $S$ and $T$, in that they will induce similar distributions over the latent variables. To quantify this, we define the following quantities. Given $S \in [K]^m \times \Dscr^m$, we will let $S = (U,B)$, where $U \in [K]^m$ and $B \in \Dscr^m$, and similarly let $T = (V,D)$ where $V\in [K]^n$ and $D \in \Dscr^n$. Likewise, given $\sigma:\Xscr \mapsto [K]\times \Dscr$, we will let $\sigma = (\mu, \beta)$, where $\mu:\Xscr \mapsto [K]$ and $\beta:\Xscr \mapsto \Dscr$, and similarly let $\tau = (\nu, \delta)$ where $\nu:\Yscr \mapsto [K]$ and $\delta:\Yscr \mapsto \Dscr$. Let $\Psi_S, \Psi_T, \Psi_\sigma,$ and $\Psi_\tau$ denote the CDFs of the values given by $S,T,\sigma$ and $\tau$, which are functions $[K] \times [0,1)^d \mapsto [0,1]$ equaling:
\begin{align*}
	\Psi_S(k,c) & = \frac{1}{m} \sum_{i=1}^m 1\{ U_i \leq k , B_i \leq c\} & \Psi_T(k,c) &= \frac{1}{n} \sum_{j=1}^n 1\{V_i \leq k, D_i \leq c\} \\
	\Psi_\sigma(k,c) &= \int_\Xscr 1\{\mu(x) \leq k, \beta(x) \leq c\}\,dx & \Psi_\tau(k,c) &= \int_\Yscr 1\{\nu(y) \leq k, \delta(y)\leq c \}\, dy,
\end{align*}
where inequalities of the form $c \leq c'$ for $c, c' \in [0,1)^d$ are satisfied if they hold entrywise.


\begin{theorem}\label{th: generalized}
Let $A \in \{0,1\}^{m \times n}$, with fixed ratio $m/n$, be generated by some $\omega$ according to Definition \ref{def: graphon}. 
Let $(\Sscr, \Tscr, \Theta)$ denote a model of the form \eqref{eq: model}.
\begin{enumerate}
\item For each $T \in \Tscr^n$, there exists $\tau_{\,T}: \Yscr \mapsto \Tscr$ such that
\begin{align} \label{eq: R_tau}
\max_{S,T,\theta \in \Sscr^m \times \Tscr^n \times \Theta} | R_A(S,T;\theta) - R_\omega(S, \tau_{\,T};\theta) - C_1| + \frac{\|\Psi_T - \Psi_{\tau_T}\|^2}{Kd}\\
\nonumber \quad { }  \leq O_P\left(d^{1/2}\left(\frac{K^2 \log n}{\sqrt{n}}\right)^{\frac{1}{1 + d}} \right),
\end{align}
where $C_1 \in \mathbb{R}$ is constant in $(S,T,\theta)$.
\item For each $S \in \Sscr^m $, there exists $\sigma_{S}:\Xscr \mapsto \Sscr$ such that
\begin{align} \label{eq: R_sigma}
\sup_{S,\tau,\theta \in \Sscr^m \times \Tscr^\Yscr \times \Theta} | R_\omega(S,\tau;\theta) - R_\omega(\sigma_{S}, \tau;\theta) - C_2| + \frac{\|\Psi_S - \Psi_{\sigma_S} \|^2}{Kd} \\
\nonumber \quad {} \leq O_P\left(d^{1/2}\left(\frac{K^2 \log n}{\sqrt{n}}\right)^{\frac{1}{1 + d}} \right),
\end{align}
where $C_2 \in \mathbb{R}$ is constant in $(S,\tau,\theta)$.
\item Combining \eqref{eq: R_tau} and \eqref{eq: R_sigma} yields
\begin{align*}
\max_{S,T,\theta \in \Sscr^m \times \Tscr^n \times \Theta}  | R_\omega(\sigma_{S}, \tau_{T}; \theta) - R_A(S,T; \theta) - C_1 - C_2 | + \frac{\|\Psi_S - \Psi_{\sigma_S}\|^2}{Kd}  \\
\quad { } + \frac{\|\Psi_T - \Psi_{\tau_T}\|^2}{Kd} =  O_P\left(d^{1/2}\left(\frac{K^2\log n}{\sqrt{n}}\right)^{\frac{1}{1 + d}} \right). 
\end{align*}
\end{enumerate}
\end{theorem}

\paragraph{Remarks for Theorem \ref{th: generalized}}

Theorem \ref{th: generalized} states that any assignment $S$ and $T$ of latent variables to the rows and columns can be extended to the populations, such that the population exhibits a similar distribution of values in $\Sscr$ and $\Tscr$, and the population risk as a function of $\theta$ is close to the empirical risk.

The theorem may also be viewed as an oracle inequality, in that for any fixed $S$ and $T$, minimizing $\theta \mapsto R_A(S,T,\theta)$ is approximately equivalent to minimizing $\theta \mapsto R_\omega(\sigma_S, \tau_{\,T}, \theta)$, as if the model $\omega$ were known. This implies that the best parametric approximation to $\omega$ can be learned, for any choice of $\sigma_S$ and $\tau_T$. However, it is not known whether the mappings $S \mapsto \sigma_S$ and $T \mapsto \tau_T$ are approximately onto; if not, minimization of $R_A$ over $(S,T,\theta)$ is a reasonable proxy for minimization of $R_\omega$ over $(\sigma,\tau,\theta)$, but only over a subset of the possible mappings $\sigma:\Xscr \mapsto \Sscr$ and $\tau:\Yscr \mapsto \Tscr$.

The convergence of $\Psi_S$ to $\Psi_{\sigma_S}$ is established in Euclidean norm. This implies pointwise convergence at every continuity point of $\Psi_{\sigma_S}$, thus implying weak convergence and also convergence in Wasserstein distance. 

The proof is contained in Appendix \ref{sec: th2 proof}. It is similar to that of Theorem \ref{th: co-blockmodel}, but requires substantially more notation due to the additional parameters. Essentially, the proof approximates the model of \eqref{eq: model} by a blockmodel, and then applies Theorem \ref{th: co-blockmodel} to bound the difference between $R_A$ and $R_\omega$.

\section{Proof of Theorem \ref{th: co-blockmodel}} \label{sec: proofs}

We present a sketch of the proof for Theorem \ref{th: co-blockmodel}, which defines the most important quantities. We then present helper lemmas and give the proof of the theorem.

\subsection{Proof Sketch} \label{sec: proof sketch}

Let $W \in [0,1]^{m \times n}$ denote the expectation of $A$, conditioned on the latent variables $x_1,\ldots,x_m$ and $y_1,\ldots,y_n$:
\[W_{ij} = \omega(x_i,y_j), \qquad i\in [m], j \in [m],\]
and let $\Phi_W(S,T)$ denote the conditional expectation of $\Phi_A(S,T)$:
\[ [\Phi_W(S,T)]_{st} = \frac{1}{nm}\sum_{i=1}^m \sum_{j=1}^n W_{ij} 1\{S_i=s, T_j=t\}.\]
Given co-cluster labels $S \in [K]^m$ and $T \in [K]^n$, let $1_{S=s} \in \{0,1\}^m$ and $1_{T=t} \in \{0,1\}^n$ denote the indicator variables
\[ 1_{S=s}(i) = \begin{cases} 1 & \text{if } S_i=s \\ 0 & \text{otherwise} \end{cases} \qquad \text{and} \qquad 1_{T=t}(j) = \begin{cases} 1 & \text{if } T_j = t \\ 0 & \text{otherwise}.\end{cases} \]
Let $g_{T=t}\in [0,1]^m$ denote the vector $n^{-1}W 1_{T=t}$, or 
\[g_{T=t}(i) = \frac{1}{n} \sum_{j=1}^n W_{ij} 1\{T_j=t\}.\]
It can be seen that the entries of $\Phi_W(S,T)$ can be written as
\begin{align}
[\Phi_W(S,T)]_{st} &= \frac{1}{m} \ip{1_{S=s}, g_{T=t}}, \label{eq: sketch1}
\end{align}
where $\ip{\cdot, \cdot}$ denotes inner product. Similarly, the entries of $\Phi_\omega(S,\tau)$ can be written as
\begin{align} \label{eq: sketch2}
 [\Phi_W(S,\tau)]_{st} = \frac{1}{m} \ip{1_{S=s}, g_{\tau=t}},
\end{align}
where $g_{\tau=t} \in [0,1]^m$ is the vector 
\[ g_{\tau=t}(i) = \int_{\Yscr} \omega(x_i,y) 1\{\tau(y) = t\}\, dy, \qquad i \in [m].\]

The proof of Theorem \ref{th: co-blockmodel} will require three main steps:
\begin{itemize}
\item[S1:] In Lemma \ref{le: mcdiarmid}, a concentration inequality will be used to show that $\Phi_A(S,T) \approx \Phi_W(S,T)$ uniformly over all possible values of $(S,T)$. 
\item[S2:] For each $T \in [K]^n$, we will show there exists $\tau:\Yscr \mapsto [K]$ such that $g_{T=t} \approx g_{\tau=t}$ for $t \in [K]$. By \eqref{eq: sketch1} and \eqref{eq: sketch2}, this will imply that $\Phi_W(S,T) \approx \Phi_\omega(S,\tau)$ uniformly for all $S \in [K]^m$. The mapping $\tau$ will also  satisfy $\pi_T \approx \pi_\tau$ as well, so that $T$ and $\tau$ have similar class frequencies.
\item[S3:] Analogous to S2, we will show that for each $S \in [K]^m$, there exists $\sigma: \Xscr \mapsto [K]$ such that $\Phi_\omega(S,\tau) \approx \Phi_\omega(\sigma_S,\tau)$ uniformly over $\tau$, and also that $\pi_S \approx\pi_{\sigma_S}$.
\end{itemize}
Steps S1 and S2 correspond to \eqref{eq: phi_tau} in Theorem \ref{th: co-blockmodel}, while step S3 corresponds to \eqref{eq: phi_sigma}.

Let $G_T$ and $G_\tau$ denote the stacked vectors in $\mathbb{R}^{mK+K}$ given by 
\begin{align*}
G_T = \left(\frac{g_{T=1}}{\sqrt{m}},\ldots,\frac{g_{T=K}}{\sqrt{m}},\pi_T\right) \qquad \text{and} \qquad  G_\tau = \left(\frac{g_{\tau=1}}{\sqrt{m}}, \ldots, \frac{g_{\tau=K}}{\sqrt{m}}, \pi_\tau\right), 
\end{align*}
and let $\Gscr_n$ and $\Gscr$ denote the set of all possible values for $G_T$ and $G_\tau$:
\begin{align*}
\Gscr_n = \{G_T: T \in [K]^n\} \qquad \text{and} \qquad \Gscr = \{G_\tau: \tau \in \Yscr \mapsto [K]\}.
\end{align*}
Step S2 is established by showing that the sets $\Gscr_n$ and $\Gscr$ converge in Hausdorff distance. This will require the following facts. The Hausdorff distance (in Euclidean norm) between two sets $\Bscr_1$ and $\Bscr_2$ is defined as
\[ d_{\textrm{Haus}}(\Bscr_1,\Bscr_2) = \max\left\{ \sup_{B_1 \in \Bscr_1} \inf_{B_2 \in \Bscr_2} \|B_1 - B_2\|, \sup_{B_2 \in \Bscr_2} \inf_{B_1 \in \Bscr_1} \| B_1 - B_2\|\right\}.\]
Given a Hilbert space $\mathbb{H}$ and a set $\Bscr \subset \mathbb{H}$, let $\Gamma_\Bscr:\mathbb{H} \mapsto \mathbb{R}$ denote the support function of $\Bscr$, defined as
\[ \Gamma_\Bscr(H) = \sup_{B\in \Bscr} \ip{H,B}.\]
It is known that the convex hull $\operatorname{conv}(\Bscr)$ equals the intersection of its supporting hyperplanes:
\[ \operatorname{conv}(\Bscr) = \left\{ x \in \mathbb{H}: \ip{x, H} \leq \Gamma_\Bscr(H)\ \textrm{for all } H \in \mathbb{H} \right\},\]
and that the Hausdorff distance between $\operatorname{conv}(\Bscr_1)$ and $\operatorname{conv}(\Bscr_2)$ is given by  \cite[Thm 1.8.11]{schneider2013convex}, \cite[Cor 7.59]{aliprantisborder}
\begin{equation}\label{eq: S2 eq1}
d_\textrm{Haus}(\operatorname{conv}(\Bscr_1), \operatorname{conv}(\Bscr_2)) = \sup_{H:\|H\|=1} | \Gamma_{\Bscr_1}(H) - \Gamma_{\Bscr_2}(H)|.
\end{equation}
To establish S2, Lemma \ref{le: supporting hyperplanes} will show that 
\begin{align} \label{eq: S2 eq2}
\sup_{H:\|H\|=1} |\Gamma_{\Gscr_n}(H) - \Gamma_{\Gscr}(H)| = O_P(K(\log n)n^{-1/2}),
\end{align}
and Lemma \ref{le: convex} will show that
\begin{equation} \label{eq: S2 eq3}
d_{\textrm{Haus}}(\operatorname{conv}(\Gscr), \Gscr) = 0.
\end{equation}
By \eqref{eq: S2 eq1} and \eqref{eq: S2 eq2}, $\operatorname{conv}(\Gscr_n)$ and $\operatorname{conv}(\Gscr)$ converge in Hausdorff distance, which by \eqref{eq: S2 eq3} implies that $\operatorname{conv}(\Gscr_n)$ and $\Gscr$ converge in Hausdorff distance. This implies that for each $G_T \in \Gscr_n$, there exists $G_{\tau} \in \Gscr$ such that $\max_T \|G_T - G_\tau\| \rightarrow 0$. This will establish S2, since $G_T \approx G_\tau$ implies by \eqref{eq: sketch1} and \eqref{eq: sketch2} that $\Phi_W(S,T) \approx \Phi_\omega(S,\tau)$ uniformly over $S \in [K]^m$, and it also implies that $\pi_T \approx \pi_{\tau}$ as well. 

The proof of S3 will be similar to S2. It can be seen that $\Phi_\omega(S,\tau)$ and $\Phi_\omega(\sigma,\tau)$ can be written as
\begin{equation} \label{eq: sketch3}
[\Phi_\omega(S,\tau)]_{st} = \ip{f_{S=s}, 1_{\tau=t}} \qquad \text{and} \qquad [\Phi_\omega(\sigma,\tau)]_{st} = \ip{f_{\sigma=s}, 1_{\tau=t}},
\end{equation}
where the functions $f_{S=s}, 1_{\tau=t}$, and $f_{\sigma=s}$ are given by
\begin{align*}
1_{\tau=t}(y) &= \begin{cases} 1 & \text{if } \tau(y) = t \\ 0 & \text{otherwise.}\end{cases} \\
f_{S=s}(y) &= \frac{1}{m}\sum_{i=1}^m \omega(x_i,y) 1\{S_i=s\} \\
f_{\sigma=s}(y) &= \int_{\Xscr} \omega(x,y) 1\{\sigma(x)=s\}\, dx.
\end{align*}
Analogous to S2, we will define sets $F_S$ and $F_\sigma$ given by
\[F_S = (f_{S=1},\ldots,f_{S=K},\pi_S) \qquad \text{and} \qquad F_\sigma  = (f_{\sigma=1},\ldots, f_{\sigma=K},\pi_\sigma),\]
whose possible values are given by
\[\Fscr_n = \{F_S: S \in [K]^m\} \qquad \text{and} \qquad \Fscr = \{F_\sigma:\sigma \in \Xscr \mapsto [K]\}. \]
Lemma \ref{le: supporting hyperplanes} will show that the support functions $\Gamma_{\Fscr_n}$ and $\Gamma_{\Fscr}$ converge, and Lemma \ref{le: convex} will show that $d_{\textrm{Haus}}(\operatorname{conv}(\Fscr), \Fscr) = 0$. Using \eqref{eq: sketch3}, this will establish S3 by arguments that are analogous to those used to prove S2.

\subsection{Intermediate Results for Proof of Theorem \ref{th: co-blockmodel}} 

Lemmas \ref{le: mcdiarmid} - \ref{le: convex} will be used to prove Theorem \ref{th: co-blockmodel}, and are proven in Section \ref{sec: lemma proofs}. 

Lemma \ref{le: mcdiarmid} states that $\Phi_A \approx \Phi_W$ for all $(S,T)$. 

\begin{lemma}\label{le: mcdiarmid}
Under the conditions of Theorem \ref{th: co-blockmodel}, 
\begin{equation}\label{eq: mcdiarmid phi}
\max_{S,T} \|\Phi_A(S,T) - \Phi_W(S,T)\|^2 = O_P\left((\log K) n^{-1}\right).
\end{equation}
\end{lemma}

Lemma \ref{le: supporting hyperplanes} states that the support functions of $\Gscr$ and $\Gscr_n$ and of $\Fscr$ and $\Fscr_n$ converge. 

\begin{lemma}\label{le: supporting hyperplanes}
Under the conditions of Theorem \ref{th: co-blockmodel}, 
\begin{align}
\sup_{\|H\|=1} \left| \Gamma_{\Gscr_n}(H) - \Gamma_{\Gscr}(H)\right| & \leq O_P(K (\log n) n^{-1/2}) \label{eq: hyperplane G}\\
\sup_{\|H\|=1} \left| \Gamma_{\Fscr_m}(H) - \Gamma_{\Fscr}(H) \right| & \leq O_P(K (\log m) m^{-1/2}), \label{eq: hyperplane F}
\end{align}
which implies
\begin{align*}
d_{\textrm{Haus}}(\operatorname{conv}(\Gscr_n), \operatorname{conv}(\Gscr)) &\leq O_P(K(\log n)n^{-1/2}) \\
d_{\textrm{Haus}}(\operatorname{conv}(\Fscr_m), \operatorname{conv}(\Fscr)) &\leq O_P(K(\log m)m^{-1/2}).
\end{align*}
\end{lemma}

Lemma \ref{le: convex} states that the sets $\Fscr$ and $\Gscr$ are essentially convex. 

\begin{lemma}\label{le: convex}
It holds that
\begin{align}
d_\textrm{Haus}(\operatorname{conv}(\Gscr), \Gscr) &= 0 \label{eq: convex eq1} \\
d_\textrm{Haus}(\operatorname{conv}(\Fscr), \Fscr) &= 0. \label{eq: convex eq2}
\end{align}
\end{lemma}

\subsection{Proof of Theorem \ref{th: co-blockmodel}}

\begin{proof}[Proof of Theorem \ref{th: co-blockmodel}]

We bound $\|\Phi_W(S,T) - \Phi_\omega(S,\tau)\|^2$ uniformly over $S$, as follows:
\begin{align}
\nonumber \|\Phi_W(S,T) - \Phi_\omega(S,\tau)\|^2 &= \sum_{s=1}^K \sum_{t=1}^K \left( [\Phi_W(S,T)]_{st} - [\Phi_\omega(S,\tau)]_{st}\right)^2 \\
\nonumber & = \sum_{s=1}^K \sum_{t=1}^K \frac{1}{m^2} \ip{1_{S=s}, g_{T=t} - g_{\tau=t}}^2 \\
\nonumber & \leq \sum_{s=1}^K \sum_{t=1}^K \frac{1}{m^2}\|1_{S=s}\|^2 \|g_{T=t} - g_{\tau = t}\|^2 \\
\nonumber & = \left( \sum_{s=1}^K \frac{1}{m} \|1_{S=s}\|^2\right) \left( \sum_{t=1}^K \frac{1}{m}\|g_{T=t} - g_{\tau = t}\|^2\right) \\
 & \leq \left( \sum_{t=1}^K \frac{1}{m}\|g_{T=t} - g_{\tau = t}\|^2\right) \label{eq: step phi_tau}
\end{align}
where \eqref{eq: step phi_tau} holds because $m^{-1}\sum_{s=1}^K \|1_{S=s}\|^2 = 1$. 

By Lemma \ref{le: supporting hyperplanes} and Lemma \ref{le: convex}, it holds that $d_{\textrm{Haus}}(\operatorname{conv}(\Gscr_n),\Gscr) = O_P(K(\log n)n^{-1/2})$. Given $T$, let $\tau \equiv \tau_{\,T}$ denote the minimizer of $\|G_T - G_\tau\| = \ip{G_T - G_\tau, G_T - G_\tau}$. It follows that 
\begin{align} 
\nonumber \max_{T} \|G_T - G_\tau\|^2 & =\max_T  \sum_{t=1}^K \frac{1}{m}\| g_{T=t} - g_{\tau = t} \|^2 + \|\pi_T - \pi_{\tau}\|^2 \\
& = O_P\left(\frac{K^2 \log n}{n}\right). \label{eq: step hausdorff G}
\end{align}

Combining \eqref{eq: mcdiarmid phi}, \eqref{eq: step hausdorff G}, and \eqref{eq: step phi_tau}  yields
\[ \max_{S,T} \|\Phi_A(S,T) - \Phi_\omega(S,\tau_{\,T})\|^2 + \|\pi_T - \pi_{\tau_{\,T}}\|^2 = O_P\left(\frac{K^2 \log n}{n}\right), \]
establishing \eqref{eq: phi_tau}.

The proof of \eqref{eq: phi_sigma} proceeds in similar fashion. The quantity $\|\Phi_\omega(S,\tau) - \Phi_\omega(\sigma,\tau)\|^2$ may be bounded uniformly over $\tau$:
\begin{align}
\nonumber \|\Phi_\omega(S,\tau) - \Phi_\omega(\sigma,\tau)\|^2 &= \sum_{s=1}^K \sum_{t=1}^K \left( [\Phi_\omega(S,\tau)]_{st} - [\Phi_\omega(\sigma,\tau)]_{st}\right)^2 \\  
\nonumber & = \sum_{s=1}^K \sum_{t=1}^K  \ip{f_{S=s} - f_{\sigma=s}, 1_{\tau=t}}^2 \\
& \leq \left( \sum_{s=1}^K \|f_{S=s} - f_{\sigma=s}\|^2\right), \label{eq: step phi_sigma}
\end{align}
where all steps parallel the derivation of \eqref{eq: step phi_tau}. It follows from Lemma \ref{le: supporting hyperplanes} and \ref{le: convex} that $d_{\textrm{Haus}}(\operatorname{conv}(\Fscr_m),\Fscr) = O_P(K(\log m)m^{-1/2})$. Given $S$, let $\sigma \equiv \sigma_S$ denote the minimizer of $\|F_S - F_\sigma\|$, so that
\begin{equation} \label{eq: step hausdorff F}
\max_{S}\quad \sum_{t=1}^K \| f_{S=s} - f_{\sigma = s} \|^2 + \|\pi_S - \pi_{\sigma}\|^2 = O_P\left(\frac{K^2 \log m}{m}\right). 
\end{equation}
Combining \eqref{eq: step hausdorff F} and \eqref{eq: step phi_sigma} yields
\[ \max_{S,\tau} \|\Phi_\omega(S,\tau) - \Phi_\omega(\sigma_S, \tau)\|^2 + \|\pi_S -\pi_{\sigma_S}\|^2 = O_P\left(\frac{K^2 \log m}{m}\right),\]
establishing \eqref{eq: phi_sigma} and completing the proof.
\end{proof}

\subsection{Proof of Lemmas \ref{le: mcdiarmid} -- \ref{le: convex}} \label{sec: lemma proofs}

The proof of Lemma \ref{le: supporting hyperplanes} will rely on Lemma \ref{le: Biau}, which is a very slight modification of Lemma 4.3 in \cite{biau2008performance}. Lemma \ref{le: Biau} is proven in the Appendix.

\begin{lemma} \label{le: Biau}
Let $\mathbb{H}$ denote a Hilbert space, with inner product $\ip{\cdot,\cdot}$ and induced norm $\|\cdot\|$. Let $g:\Yscr \mapsto \mathbb{H}$, and let $y_1,\ldots,y_n \in \Yscr$ be i.i.d. Let $L_n: \mathbb{H}^K \mapsto \mathbb{R}$ be defined as
\begin{equation}\label{eq: L_n}
L_n(H) = \frac{1}{n}\sum_{j=1}^n \max_{k \in [K]} \ip{h_k, g(y_j)}, \qquad H=(h_1,\ldots,h_K) \in \mathbb{H}^K.
\end{equation}
Let $\Hscr = \left\{H \in \mathbb{H}^K : \|h_k\| \leq 1, t \in [K]\right\}$. It holds that
\begin{equation*} 
\E \sup_{H \in \Hscr} | L_n(H) - \E L_n(H)| \leq 2 K \left(\frac{\mathbb{E}\|g(y)\|^2}{n}\right)^{1/2}.
\end{equation*}
\end{lemma}

To prove Lemma \ref{le: convex}, we will require a theorem for finite dimensional convex hulls:

\begin{theorem} \label{th: caratheodory} \cite[Thm 1.1.4]{schneider2013convex}
If $\Bscr \subset \mathbb{R}^d$ and $x \in \operatorname{conv}(\Bscr)$, there exists $B_1,\ldots,B_{d+1}$ such that $x \in \operatorname{conv}\{B_1,\ldots,B_{d+1}\}$. 
\end{theorem}

Additionally, we will also require some results on Hilbert-Schmidt integral operators. A kernel function $\omega:\Xscr \times \Yscr \mapsto \mathbb{R}$ is Hilbert-Schmidt if it satisfies
\[\int_{\Xscr \times \Yscr} |\omega(x,y)|^2 dx dy < \infty.\]
It can be seen that $\omega$ defined by Definition \ref{def: graphon} is Hilbert-Schmidt. Let $\Omega$ denote the integral operator induced by $\omega$, given by
\[(\Omega f)(x) = \int_{\Yscr} \omega(x,y) f(y) dy.\]
It is known that a Hilbert-Schmidt operator $\Omega$ is a limit (in operator norm) of a sequence of finite rank operators, so that its kernel $\omega$ has singular value decomposition given by
\[ \omega(x,y) = \sum_{q=1}^\infty \lambda_q u_q(x) v_q(y),\]
where $\{u_q\}_{q=1}^\infty$ and $\{v_q\}_{q=1}^\infty$ are sets of orthonormal functions mapping $\Xscr \mapsto \mathbb{R}$ and $\Yscr \mapsto \mathbb{R}$, and $\lambda_1,\lambda_2,\ldots$ are scalars decreasing in magnitude and satisfying $\sum_{q=1}^\infty \lambda_q^2 < \infty$.

\begin{proof}[Proof of Lemma \ref{le: mcdiarmid}]
Given $(S,T)$, let $\Delta \in [-1,1]^{K\times K}$ denote the quantity
\[\Delta_{st} = \frac{1}{mn} \sum_{i=1}^m \sum_{j=1}^n (A_{ij} - W_{ij})1(S_i=s, T_j=t).\]
It holds that $\mathbb{E}[\Delta | W] = 0$, and by Hoeffding's inequality,
\[\mathbb{P}\left( |\Delta_{st}| \geq \epsilon | W\right) \leq 2 e^{-2nm\epsilon^2},\qquad s,t \in [K].\]
Conditioned on $W$, each entry of $\Delta$ is independent of the others. Given $\delta \in [-1,1]^{K\times K}$, it follows that
\begin{align*}
\mathbb{P}\left( \Delta = \delta | W \right) & = \prod_{s=1}^K \prod_{t=1}^K \mathbb{P}\left(\Delta_{st} = \delta_{st}\, | W\right) \\
& \leq 2 \exp\left(-2nm \sum_{s=1}^K \sum_{t=1}^K \delta_{st}^2\right).
\end{align*}
Let $B$ denote the set 
\[ B = \left\{\delta \in [-1,1]^{K\times K}: \sum_{s,t} \delta_{st}^2 \geq \epsilon, \delta \in \operatorname{supp}(\Delta) \right\}.\]
The cardinality of $B$ is smaller than the support of $\Delta$, which is less than $(nm)^{K^2}$ when conditioned on $W$. It follows by a union bound over $B$ that 
\begin{align*}
\mathbb{P}\left( \Delta \in B|W\right) &\leq 2|B| e^{-2nm\epsilon} \\
& \leq 2(nm)^{K^2} e^{-2nm\epsilon}.
\end{align*}
It can be seen that $\|\Phi_A(S,T) - \Phi_W(S,T)\|^2 = \sum_{s,t} \Delta_{st}^2$, implying that $\Delta \in B$ is equivalent to the event that $\|\Phi_A(S,T) - \Phi_W(S,T)\|^2 \geq \epsilon$. A union bound over all $S,T$ implies that 
\[ \mathbb{P}\left( \max_{S,T} \|\Phi_A(S,T) - \Phi_W(S,T)\|^2 \geq \epsilon\right) \leq 2K^{n+m} (nm)^{K^2}e^{-2nm\epsilon}.\]
Letting $\epsilon = C(1+n/m)(\log K) n^{-1}$ for some $C$ proves the lemma.
\end{proof}

\begin{proof}[Proof of Lemma \ref{le: supporting hyperplanes}]
Let $g_y \in [0,1]^m$ denote the column of $W$ induced by $y \in \Yscr$, and let $f_x \in [0,1]^\Yscr$ denote the row of $\omega$ corresponding to $x \in \Xscr$:
\[ g_y(i) = \omega(x_i,y), \qquad i \in [m] \qquad \text{and} \qquad f_x(y) = \omega(x,y), \qquad y \in \Yscr.\]
Algebraic manipulation shows that $g_{T=t}, g_{\tau=t}, f_{S=s},$ and $f_{\sigma=s}$ can be written as
\begin{align*}
g_{T=t}  &= \frac{1}{n} \sum_{j=1}^n g_{y_j}1(T_j=t) & g_{\tau=t} &= \int_{\Yscr} g_y 1(\tau(y)=t)\, dy \\
f_{S=s}  &= \frac{1}{m} \sum_{i=1}^m f_{x_i}1(S_i=s) & f_{\sigma = s} &= \int_{\Xscr} f_x 1(\sigma(x)=s)\, dx.
\end{align*}
Given $H = (h_1,\ldots,h_K, \pi_H)$, it follows that the inner products $\ip{H,G_T}, \ip{H, G_\tau}, \ip{H,F_S}$, and $\ip{H,F_\sigma}$  equal
\begin{align*}
\ip{H, G_T} &= \frac{1}{n}\sum_{j=1}^n \left[\ip{h_{T_j}, \frac{g_{y_j}}{\sqrt{m}}} + \pi_H(T_j)\right], & 
\hskip.5cm \ip{H, G_\tau} &= \int_{\Yscr} \ip{h_{\tau(y)}, \frac{g_y}{\sqrt{m}}} + \pi_H(\tau(y))\, dy \\
\ip{H, F_S} &= \frac{1}{m}\sum_{i=1}^m \left[\ip{h_{S_i},f_{x_i}} + \pi_H(S_i)\right], &
\ip{H, F_\sigma} &= \int_{\Xscr} \ip{h_{\sigma(x)}, f_x} + \pi_H(\sigma(x))\, dx,
\end{align*}
and hence that the support functions equal
\begin{align*} 
\Gamma_{\Gscr_n}(H) &= \frac{1}{n} \sum_{j=1}^n \max_{k\in [K]} \ip{h_k,\frac{g_{y_j}}{\sqrt{m}}} + \pi_H(k), & 
\hskip.5cm \Gamma_{\Gscr}(H) &= \int_{\Yscr} \max_{k\in [K]} \ip{h_k, \frac{g_y}{\sqrt{m}}} + \pi_H(k)\, dy \\
\Gamma_{\Fscr_m}(H) &= \frac{1}{m} \sum_{i=1}^m \max_{k \in [K]} \ip{h_k, f_{x_i}} +\pi_H(k), &
\Gamma_{\Fscr}(H) &= \int_{\Xscr} \max_{k \in [K]} \ip{h_k, f_x} + \pi_H(k)\, dx,
\end{align*}
which implies that $\mathbb{E} \Gamma_{\Gscr_n}(H) = \Gamma_\Gscr(H)$ and $\mathbb{E} \Gamma_{\Fscr_m}(H) = \Gamma_{\Fscr}(H)$. 

To show \eqref{eq: hyperplane G}, we observe that $\Gamma_{\Gscr_n}$ can be rewritten as
\[ \Gamma_{\Gscr_n}(H) = \frac{1}{n} \sum_{j=1}^n \max_{k \in [K]} \ip{\left[\begin{array}{c}h_k \\ \pi_H(k)\end{array}\right], \left[\begin{array}{c} m^{-1/2} g_{y_j} \\ 1 \end{array}\right]},\]
which matches \eqref{eq: L_n} so that Lemma \ref{le: Biau} can be applied. Applying Lemma \ref{le: Biau} results in
\begin{equation}\label{eq: proof hyperplanes}
\mathbb{E}\sup_{\|H\| = 1} \left| \Gamma_{\Gscr_n}(H) - \Gamma_\Gscr(H)\right| \leq \frac{4K}{\sqrt{n}},
\end{equation}
where we have used $\{H:\|H\| = 1\} \subset \Hscr$ and $\left\| \left[\begin{array}{c} m^{-1/2} g_{y_j} \\ 1\end{array}\right] \right\|^2 \leq 2$.

Let $Z(y_1,\ldots,y_n) = \sup_{\|H\|=1} | \Gamma_{\Gscr_n}(H) - \Gamma_\Gscr(H)|$. For $\ell \in [n]$, changing $y_\ell$ to $y_\ell'$ changes $Z$ by at most $4/n$. Applying McDiarmid's inequality yields
\[\mathbb{P}\left( |Z - \mathbb{E}Z| \geq \epsilon\right) \leq 2e^{-2\epsilon^2n/8}.\]
Letting $\epsilon = n^{-1/2}\log n$ implies that $Z - \mathbb{E}Z = O_P(n^{-1/2} \log n)$, which combined with \eqref{eq: proof hyperplanes} implies \eqref{eq: hyperplane G}. 

To show \eqref{eq: hyperplane F}, we observe that 
\[ \Gamma_{\Fscr_m}(H) = \frac{1}{m} \sum_{i=1}^m \max_{k \in [K]} \ip{\left[\begin{array}{c}h_k \\ \pi_H(k)\end{array}\right], \left[\begin{array}{c} f_{x_i} \\ 1 \end{array}\right]},\]
so that Lemma \ref{le: Biau} and McDiarmid's inequality can be used analogously to the proof of \eqref{eq: hyperplane G}.

\end{proof}

We divide the proof of Lemma \ref{le: convex} into two sub-lemmas, one showing \eqref{eq: convex eq1} and the other showing \eqref{eq: convex eq2}. This is because the proof of \eqref{eq: convex eq2} will require additional work, due to the fact that the elements of $\Fscr$ are infinite dimensional.

\begin{lemma}\label{le: convex eq1}
For each $G^* \in \operatorname{conv}(\Gscr)$, there exists $G_1,G_2,\ldots \in \Gscr$ such that $\lim_{\ell\rightarrow \infty} \|G^* - G_\ell\| = 0.$
\end{lemma}

\begin{lemma} \label{le: convex eq2}
For each $F^* \in \operatorname{conv}(\Fscr)$, there exists $F_1,F_2,\ldots \in \Fscr$ such that $\lim_{\ell \rightarrow \infty} \|F^* - F_\ell\| = 0.$
\end{lemma}

\begin{proof}[Proof of Lemma \ref{le: convex eq1}]

Recall the definition of $g_y \in [0,1]^m$ as defined in the proof of Lemma \ref{le: Biau}:
\[g_y(i) = \omega(x_i,y), \qquad i \in [m],\]
and that $g_{\tau=t}$ can be written as
\[g_{\tau=t} = \int_{\Yscr} g_y1\{\tau(y) = t\}\, dy.\]
We note the following properties of $\{g_y: y \in \Yscr\}$:
\begin{enumerate}
\item[P1:] Each $G^* \in \operatorname{conv}(\Gscr)$ is a finite convex combination of elements in $\Gscr$. This holds by Theorem \ref{th: caratheodory}, since $\Gscr$ is a subset of $[0,1]^{mK + K}$, a finite dimensional space.
\item[P2:] For all $\epsilon$, there exists a finite set $\Bscr$ that is an $\epsilon$-cover of $\{g_y:y \in \Yscr\}$ in Euclidean norm. This holds because $\{g_y:y \in \Yscr\}$ is a subset of the unit cube $[0,1]^m$.
\end{enumerate}

By P1, each $G^* \in \operatorname{conv}(\Gscr)$ can be written as a finite convex combination of elements in $\Gscr$, so that for some integer $N>0$ there exists $G_{\tau_1}, \ldots,G_{\tau_N} \in \Gscr$ such that
\[ G^* = \sum_{i=1}^N \eta_i G_{\tau_i},\]
where $\eta$ is in the $N$-dimensional unit simplex. It follows that for some $\mu:\Yscr \mapsto [0,1]^K$ satisfying $\sum_k \mu_k(y) = 1$ for all $y$, $G^* \equiv (g^*_1,\ldots, g^*_K, \pi_G^*)$ satisfies
\begin{align*} 
g^*_k = \int_\Yscr g_y \mu_k(y) dy \qquad \text{and} \qquad \pi_G^*(k) = \int_\Yscr \mu_k(y) dy, \quad k\in[K].
\end{align*}
We now construct $\tau:\Xscr \mapsto [K]$ inducing $G_\tau \in \Gscr$ which approximates $G^* \in \operatorname{conv}(\Gscr)$. By P2, let $\Bscr$ denote an $\epsilon$-cover of $\{g_y:y \in \Yscr\}$, and enumerate its elements as $b_1,\ldots, b_{|\Bscr|}$. For each $y \in \Yscr$, let $\ell:\Yscr \mapsto [|\Bscr|]$ assign $y$ to its closest member in $\Bscr$, so that $\|g_y - b_{\ell(y)}\| \leq \epsilon$. For $i = 1,\ldots,|\Bscr|$, let $\Yscr_i$ denote the set $\{y: \ell(y) = i\}$. Arbitrarily divide each region $\Yscr_i$ into $K$ disjoint sub-regions $\Yscr_{i1},\ldots,\Yscr_{iK}$ 
such that $\cup_{k} \Yscr_{ik} = \Yscr_i$, where the measure of each sub-region is given by
\begin{equation}\label{eq: convex y_ik}
\int_{\Yscr_{ik}} 1\, dy = \int_{\Yscr_i} \mu_k(y) dy, \qquad k \in [K].
\end{equation}
Let $\tau:\Yscr \mapsto [K]$ assign each region $\Yscr_{ik}$ to $k$, so that
\[\tau(y) = k \text{ for all $y \in \Yscr_{ik}, i=1,\ldots, |\Bscr|$}.\]
By \eqref{eq: convex y_ik}, it holds that $\pi_\tau = \pi_G^*$, and also that
\begin{align*}
g_{\tau=k} - g^*_k &= \int_\Yscr g_y \left[1\{\tau(y) = k\} - \mu_k(y)\right]\,dy  \\
&= \int_{\Yscr} \left[b_{\ell(y)} + g_y - b_{\ell(y)}\right] \left[1\{\tau(y) = k\} - \mu_k(y)\right]\,dy  \\
&= \sum_{i=1}^{|\Bscr|} b_i \underbrace{\left[\int_{\Yscr_{ik}} 1 \,dy - \int_{\Yscr_i} \mu_k(y)\,dy\right]}_{=0 \text{ by } \eqref{eq: convex y_ik}} + \int_\Yscr (g_y - b_\ell(y)) \left[1\{\tau(y) = k\} - \mu_k(y)\right]\,dy \\
&= 0 + \int_\Yscr (g_y - b_{\ell(y)}) \left[1\{\tau(y) = k\} - \mu_k(y)\right]\,dy,
\end{align*}
which implies that
\begin{align*}
\|g_{\tau=k} - g^*_k\| & \leq \left\|\int_\Yscr (g_y - b_{\ell(y)}) 1\{\tau(y)=k\}\, dy\right\| + \left\|\int_\Yscr (g_y - b_{\ell(y)}) \mu_k(y)\, dy\right\| \\
& \leq 2 \int_\Yscr \|g_y - b_{\ell(y)}\|\,dy \\
& \leq 2 \epsilon.
\end{align*}
It follows that $\|G_\tau - G^*\|^2 = \sum_{k=1}^K m^{-1}\|g_{\tau=k} - g^*_k\|^2 + \|\pi_\tau - \pi_G^*\|^2 \leq 4K\epsilon^2 m^{-1}$, and hence that $\lim_{\epsilon\rightarrow 0} \|G_\tau - G^*\| = 0$, proving the lemma.

\end{proof}

\begin{proof}[Proof of Lemma \ref{le: convex eq2}]
Recall the definition of $f_x: \Yscr \mapsto [0,1]$ as defined in the proof of Lemma \ref{le: Biau}:
\[f_x(y) = \omega(x,y),\]
and that $f_{\sigma=s}$ can be written as
\[f_{\sigma=s} = \int_\Xscr f_x 1\{\sigma(x)=s\}\, dx.\]
Because $\{f_x:x \in \Xscr\}$ is not finite dimensional, the arguments of Lemma \ref{le: convex eq1} do not directly apply. To circumvent this, we will approximate the space $\Fscr$ by a finite dimensional $\hat{\Fscr}$, such that the convex hulls $\operatorname{conv}(\Fscr)$ and $\operatorname{conv}(\hat{\Fscr})$ converge.

For $Q=1,2,\ldots,$ let $\omega_Q$ be the best rank-Q approximation to $\omega$,
\[\omega_Q(x,y) = \sum_{q=1}^Q \lambda_q u_q(x) v_q(y).\]
Given $D>0$, let $\hat{u}_q$ denote a truncation of $u_q$, defined as 
\[ \hat{u}_q^D(x) = \begin{cases} D & \text{if } u_q(x) \geq D \\ u_q(x) & \text{if } -D \leq u_q(x) \leq D \\ -D & \text{if } u_q(x) \leq -D, \end{cases} \]
and let $\hat{\omega}:\Xscr \times \Yscr \mapsto \mathbb{R}$ be defined as
\[ \hat{\omega}(x,y) = \sum_{q=1}^Q \lambda_q \hat{u}_q(x) v_q(y).\]
Let $\hat{f}_x:\Yscr \mapsto \mathbb{R}$ and $\hat{f}_{\sigma=s}$ be defined as
\[ \hat{f}_x(y) = \hat{\omega}(x,y) \qquad \text{and} \qquad \hat{f}_{\sigma=s} = \int_\Xscr \hat{f}_x 1\{\sigma(x)=s\}\, dx.\]
Let $\hat{F}_\sigma$ and $\hat{\Fscr}$ be defined as
\[ \hat{F}_\sigma = (\hat{f}_{\sigma=1},\ldots,\hat{f}_{\sigma=K}, \pi_\sigma) \qquad \text{and} \qquad \hat{\Fscr} = \{\hat{F}_\sigma: \sigma \in [K]^\Xscr\}.\]
We bound the difference $\| \hat{f}_x - f_x\|^2$:
\begin{align*}
\| \hat{f}_x - f_x\|^2 = \sum_{q=1}^Q \lambda_q^2 (\hat{u}_q(x) - u_q(x))^2 + \sum_{q=Q+1}^\infty \lambda_q^2 u_q(x)^2 ,
\end{align*}
where we used the fact $f_x = \sum_{q=1}^\infty \lambda_q u_q(x) v_q$, and that the functions $\{v_q\}$ are orthonormal. It follows that
\begin{align*}
\int_{\Xscr} \| \hat{f}_x - f_x\|^2\, dx &= \sum_{q=1}^Q \lambda_q^2 \int_{\Xscr} (\hat{u}_q(x) - u_q(x))^2\, dx + \sum_{q=Q+1}^\infty \lambda_q^2 \int_\Xscr u_q(x)^2\, dx\\
&= \sum_{q=1}^Q \lambda_q^2 \int_\Xscr (\hat{u}_q(x) - u_q(x))^2 dx + \sum_{q=Q+1}^\infty \lambda_q^2 \\
& \leq \sum_{q=1}^Q \lambda_q^2 \int_{x:|u_q(x)|\geq D} u_q(x)^2\, dx + \sum_{q=Q+1}^\infty \lambda_q^2, 
\end{align*}
from whence it can be seen that
\[ \lim_{\min(Q,D) \rightarrow \infty} \int_{\Xscr} \| \hat{f}_x - f_x\|^2\, dx = 0.\]
We use this result to bound $\|\hat{f}_{\sigma=s} - f_{\sigma=s}\|$:
\begin{align*}
\max_{s,\sigma} \|\hat{f}_{\sigma=s} - f_{\sigma=s}\|^2 &= \max_{s,\sigma} \left\|\int_\Xscr(\hat{f}_x - f_x)1_{\sigma=s}(x)\, dx \right\|^2\\
&\leq \int_\Xscr \| \hat{f}_x - f_x\|^2\, dx \\
& \rightarrow 0 \text{ as $\min(Q,D) \rightarrow \infty$}.
\end{align*}
Since $\|\hat{F}_\sigma - F_\sigma\|^2 = \sum_{k=1}^K \|\hat{f}_{\sigma=k} - f_{\sigma=k}\|^2 + \|\pi_\sigma - \pi_\sigma\|^2$, it follows that for any $\epsilon > 0$, there exists $(Q,D)$ inducing $\hat{\Fscr} = \{\hat{F}_\sigma: \sigma \in [K]^\Xscr\}$ such that 
\begin{equation} \label{eq: convex eq2 basic}
 \sup_{\sigma} \|\hat{F}_\sigma - F_\sigma\| \leq \epsilon,
\end{equation}
so that the support functions of $\Fscr$ and $\hat{\Fscr}$ can be bounded by
\begin{align*}
\sup_{H:\|H\| = 1} |\Gamma_{\Fscr}(H) - \Gamma_{\hat{\Fscr}}(H)| & \leq \max_{\|H\|=1,\sigma}\left| \ip{H, F_\sigma - \hat{F}_\sigma} \right|\\
& \leq \max_\sigma \|F_\sigma - \hat{F}_\sigma\| \\
& \leq \epsilon,
\end{align*}
implying that 
\begin{equation} \label{eq: convex eq2 hausdorff}
 d_{\textrm{Haus}}(\operatorname{conv}(\Fscr)), \operatorname{conv}(\hat{\Fscr})) \leq \epsilon,
\end{equation}
which in turn implies that for any $F^* \in \operatorname{conv}(\Fscr)$, there exists $\hat{F}^* \in \operatorname{conv}(\hat{\Fscr})$ such that $\| F^* - \hat{F}^* \| \leq \epsilon$.

For any choice of $(Q,D)$, we observe that properties P1 and P2 as described in Lemma \ref{le: convex eq1} for $\Gscr$ also hold for $\hat{\Fscr}$:
\begin{enumerate}
\item[P1:] Each $\hat{F} \in \operatorname{conv}(\hat{\Fscr})$ is a finite convex combination of elements in $\hat{\Fscr}$. This holds because each $\hat{f}_x$ can be written as
\[ \hat{f}_x = \sum_{q=1}^Q \lambda_q \hat{\mu}_q(x) v_q,\]
showing that $\{\hat{f}_x:x \in \Xscr\}$ is a finite dimensional subspace of $\Yscr \mapsto \mathbb{R}$, and hence $\hat{\Fscr}$ is as well, allowing Theorem \ref{th: caratheodory} to be applied.
\item[P2:] For all $\epsilon$, there exists a finite $\epsilon$-cover of $\{\hat{f}_x:x \in \Xscr\}$ in Euclidean norm. This holds because the set $\{\hat{u}(x): x \in \Xscr\}$ is a subset of the hypercube $[-D,D]^Q$.
\end{enumerate}
As a result, the same arguments used to prove Lemma \ref{le: convex eq1} also apply to $\hat{\Fscr}$, implying that for each $\hat{F} \in \operatorname{conv}(\hat{\Fscr})$, there exists for any $\epsilon>0$ a mapping $\sigma:\Xscr \mapsto [K]$ such that 
\begin{equation}\label{eq: convex eq2 eq1}
\|\hat{F}_\sigma - \hat{F}\|^2 \leq 4K\epsilon^2. 
\end{equation}
It thus follows that for any $\epsilon > 0$ and $F^* \in \operatorname{conv}(\Fscr)$, there exists $\hat{F}^* \in \operatorname{conv}(\hat{\Fscr})$ and $\sigma:\Xscr \mapsto [K]$ such that
\begin{align*}
\|F^* - F_\sigma\| &\leq \underbrace{\|F^* - \hat{F}^*\|}_{\leq \epsilon \text{ by } \eqref{eq: convex eq2 hausdorff}} + \underbrace{\|\hat{F}^* - \hat{F}_\sigma\|}_{\leq 4K\epsilon^2 \text{ by } \eqref{eq: convex eq2 eq1}} + \underbrace{\|\hat{F}_\sigma - F_\sigma\|}_{\leq \epsilon \text{ by } \eqref{eq: convex eq2 basic}} \\
& \leq 2\epsilon + 4\epsilon^2K.
\end{align*}
As a result, it follows that there exists $F_1,F_2,\ldots \in \Fscr$ such that $\lim_{i \rightarrow \infty} \|F^* - F_i \| = 0$.
\end{proof}

\begin{proof}[Proof of Lemma \ref{le: convex}]
Lemma \ref{le: convex} follows immediately from Lemmas \ref{le: convex eq1} and \ref{le: convex eq2}, which establish \eqref{eq: convex eq1} and \eqref{eq: convex eq2} respectively.
\end{proof}


\appendix

\section{Proof of Lemma 4} 

To prove Lemma \ref{le: Biau}, we will use a result from \cite{biau2008performance}, which we state and prove here:

\begin{lemma}\cite[Lemma 4.3]{biau2008performance} \label{le: old Biau}
Let $\mathbb{H}$ denote a Hilbert space, and let $g:\Yscr \mapsto \mathbb{H}$. Let $y_1,\ldots,y_n \in \Yscr$ be i.i.d, and let $L_n:\mathbb{H}^K \mapsto \mathbb{R}$ be defined as follows:
\[L_n(H) = \frac{1}{n}\sum_{j=1}^n \max_{t \in [K]} \ip{h_t, g(y_j)}, \qquad H = (h_1,\ldots,h_K) \in \mathbb{H}^K\]
Let $\Bscr = \{H \in \mathbb{H}^K: \|h_k\|\leq 1, k \in [K]\}$. Then the following three statements hold:
\begin{equation} \label{eq: rademacher}
\E \sup_{H \in \Bscr} L_n(H) - \E L_n(H) \leq 2 \E \sup_{H \in \Bscr} \frac{1}{n}\sum_{j=1}^n \epsilon_j \max_{t\in [K]} \ip{h_t, g(y_j)} ,
\end{equation}
where $\epsilon_1,\ldots,\epsilon_j \stackrel{iid}{\sim} \pm 1$ w.p. $1/2$,
\begin{equation} \label{eq: induction}
\E \sup_{H \in \Bscr} \frac{1}{n} \sum_{j=1}^n \epsilon_j \max_{t \in [K]} \ip{h_t, g(y_j)}  \leq 2 K \E \sup_{\|h\| = 1} \frac{1}{n} \sum_{j=1}^n \epsilon_j \ip{h, g(y_j)},
\end{equation}
and
\begin{equation} \label{eq: base case}
\E \sup_{\|h\|=1} \frac{1}{n} \sum_{j=1}^n \epsilon_i \ip{h, g(y_j)} \leq \left(\frac{\mathbb{E} \|g(y)\|^2}{n}\right)^{1/2}.
\end{equation}
\end{lemma}

\begin{proof}[Proof of Lemma \ref{le: old Biau}]

\eqref{eq: rademacher} is a standard symmetrization argument \cite{bousquet2004introduction}. Letting $y_1',\ldots,y_j'$ denote i.i.d Uniform $[0,1]$ random variables, and $\epsilon_1,\ldots, \epsilon_n \stackrel{iid}{\sim} \pm 1$ w.p. $1/2$, it holds that
\begin{align*}
\E \sup_{H \in \Bscr} L_n(H) - \E L_n(H) &\leq \E \sup_{H \in \Bscr} \frac{1}{n} \sum_{j=1}^n \max_{t \in [K]} \ip{h_t, g(y_j)} - \max_{t \in [K]} \ip{h_t, g(y_j')} \\
& = \E \sup_{H \in \Bscr} \frac{1}{n} \sum_{j=1}^n \epsilon_i \left( \max_{t \in [K]} \ip{h_t, g(y_j)} - \max_{t \in [K]} \ip{h_t, g(y_j')} \right) \\
& \leq \E \sup_{H \in \Bscr} \frac{1}{n} \sum_{j=1}^n \epsilon_i \max_{t \in [K]} \ip{h_t, g(y_j)} + \E \sup_{H \in \Bscr} \frac{1}{n} \sum_{j=1}^n \epsilon_j \max_{t \in [K]} \ip{h_t, g(y_j')} \\
& = 2 \E \sup_{H \in \Bscr} \frac{1}{n} \sum_{j=1}^n \epsilon_i \max_{t \in [K]} \ip{h_t, g(y_j)}.
\end{align*}
To show \eqref{eq: induction}, let $\Rscr(\Fscr)$ denote the (non-absolute valued) Rademacher complexity of a function class $\Fscr$:
\[\Rscr(\Fscr) = \E \sup_{f \in \Fscr} \frac{1}{n}\sum_{j=1}^n \epsilon_j f(y_j).\]
The following contraction principles for Rademacher complexity hold: \cite{biau2008performance, ledoux2013probability}
\begin{enumerate}
\item $\Rscr(|\Fscr|) \leq \Rscr(\Fscr)$, where $|\Fscr| = \{|f|: f \in \Fscr\}$. \cite[Thm 11.6]{boucheron2013concentration} 
\item $\Rscr(\Fscr_1 \oplus \Fscr_2) \leq \Rscr(\Fscr_1) + \Rscr(\Fscr_2)$, where $\Fscr_1 \oplus \Fscr_2 = \{f_1 + f_2: (f_1,f_2) \in \Fscr_1 \times \Fscr_2\}$.
\end{enumerate}
For $K=2$, \eqref{eq: induction} follows from the following steps,
\begin{align*}
\E \sup_{H \in \Bscr} \frac{1}{n} \sum_{j=1}^m \epsilon_j \max_{t \in [2]} \ip{h_t, g(y_j)} & = \frac{1}{2}\E\bigg\{ \sup_{H \in \Bscr}  \frac{1}{n} \sum_{j=1}^n \epsilon_j \bigg[ \ip{h_1, g(y_j)} + \ip{h_2, g(y_j)} \\
& \qquad {} + |\ip{h_1, g(y_j)} - \ip{h_2, g(y_j)}|\, \bigg] \bigg\}\\
& = \E\left\{ \sup_{\|h_1\|=1}  \frac{1}{n} \sum_{j=1}^n \epsilon_j \ip{h_1, g(y_j)} + \sup_{\|h_2\|=1} \frac{1}{n} \sum_{j=1}^n \epsilon_j \ip{h_2, g(y_j)} \right\} \\
& = K \E \sup_{\|h\|=1} \frac{1}{n} \sum_{j=1}^n \epsilon_j \ip{h, g(y_j)}, 
\end{align*}
which holds by $\max(a,b) = (a + b + |a-b|)/2$ and the contraction principles. The induction rule for general $K$ is straightforward, using the fact that $\max(a_1,\ldots,a_K) = \max(\max(a_1,\ldots,a_{K-1}), a_K)$.

To show \eqref{eq: base case}, observe that
\begin{align*}
\mathbb{E} \sup_{\|h\|=1} \frac{1}{n} \sum_{j=1}^n \epsilon_j \ip{h, g(y_j)} & = \mathbb{E} \sup_{\|h\|=1} \ip{h, \frac{1}{n} \sum_{j=1}^n \epsilon_j g(y_j)}  \\
& = \mathbb{E} \left\| \frac{1}{n} \sum_{j=1}^n \epsilon_j g(y_j) \right\| \\
& \leq \left(\mathbb{E} \left\| \frac{1}{n} \sum_{j=1}^n \epsilon_j g(y_j) \right\|^2 \right)^{1/2} \\
& = \left(\frac{1}{n} \mathbb{E} \|g(y_1)\|^2 \right)^{1/2}.
\end{align*}
\end{proof}

\begin{proof}[Proof of Lemma \ref{le: Biau}]
\eqref{eq: rademacher} - \eqref{eq: base case} imply that 
\begin{equation}\label{eq: positive}
\E \sup_{H \in \Bscr} L_n(H) - \E L_n(H) \leq K  \left(\frac{\mathbb{E}\|g(y)\|^2}{n}\right)^{1/2}.
\end{equation}
It also holds that 
\begin{align}
\nonumber \E \inf_{H \in \Bscr} L_n(H) - \E L_n(H) & \geq 2 \E \inf_{H \in \Bscr} \frac{1}{n} \sum_{j=1}^n \epsilon_j \max_{t \in [K]} \ip{h_t, g(y_j)} \\
\nonumber  &= -2\E \sup_{H \in \Bscr} \frac{1}{n}\sum_{j=1}^n (-\epsilon_j) \max_{t \in [K]} \ip{h_t, g(y_j)} \\
\nonumber & = -2\E \sup_{H \in \Bscr} \frac{1}{n} \sum_{j=1}^n \epsilon_j \max_{t \in [K]} \ip{h_t, g(y_j)} \\
 & \geq - 2 K \left(\frac{\mathbb{E}\|g(y)\|^2}{n}\right)^{1/2}, \label{eq: negative}
\end{align}
where the first inequality holds by a symmetrization analogous to \eqref{eq: rademacher}; the second by algebraic manipulation; the third because $\epsilon_1,\ldots,\epsilon_n$ are $\pm 1$ with probability $1/2$; the fourth by \eqref{eq: induction} and \eqref{eq: base case}.

Combining \eqref{eq: positive} and \eqref{eq: negative} proves the lemma.
\end{proof}

\section{Proof of Theorem 2} 
\label{sec: th2 proof}

\subsection*{Preliminaries}

Let $\Dscr = \{c \in [0,1)^d: \|c\|\leq 1\}$, $\Sscr = [K] \times \Dscr,  \Tscr = [K] \times \Dscr$, and $\Theta = [0,1]^{K \times K}$. Let $\bar{\Dscr}$ denote the smallest $\epsilon$-cover in 2-norm of $\Dscr$. Let $\bar{\Sscr} = [K] \times \bar{\Dscr}$ and let $\bar{\Tscr} = [K] \times \bar{\Dscr}$. Let $\bar{K} = |\bar{\Sscr}| = |\bar{\Tscr}| \leq K (\sqrt{d}\epsilon^{-1})^d$. 

As described in Section \ref{sec: th2}, recall that we may write $S, T, \sigma$ and $\tau$ as $S = (U,B), T=(V,D), \sigma = (\mu,\beta)$, and $\tau = (\nu,\delta)$. Given $S = (U,B)$, let $\bar{S}$ denote its closest approximation in $\bar{\Sscr}^m$. This means that $\bar{S} = (U, \bar{B})$, with $\bar{B} \in \bar{\Dscr}^m$ satisfying $\bar{B}_i = \arg \min_{c\, \in \bar{\Dscr}} \|B_i - c\|$ for $i \in [m]$. Similarly, given $T = (V,D)$ or $\tau = (\nu,\delta)$, let $\bar{T} = (V, \bar{D})$ or $\bar{\tau} = (\nu, \bar{\delta})$ be defined analogously.


Let $Z \in [0,1]^{m \times n}$ be defined by
\[ Z_{ij} = \bar{B}_i^T \bar{D}_j W_{ij},\]
and let $\Phi_Z(U,V)$ be defined by
\[ [\Phi_Z(U,V)]_{uv} = \frac{1}{mn} \sum_{i=1}^m \sum_{j=1}^n Z_{ij} 1\{U_i=u, V_j=v\}.\]
Let $\Phi_\zeta(U, \nu)$ and $\Phi_\zeta(\mu, \nu)$ denote population versions of $\Phi_Z$, defined by
\begin{align*}
[\Phi_\zeta(U, \nu)]_{uv} &= \frac{1}{m} \sum_{i=1}^m \int_\Yscr \bar{B}_i^T \delta(y) \omega(x_i, y) 1\{U_i=u, \nu(y)=v\}\, dy, \\
[\Phi_\zeta(\mu, \nu)]_{uv} &= \int_{\Xscr \times \Yscr} \bar{\beta}(x)^T \bar{\delta}(y)\, \omega(x,y) 1\{\mu(x)=u, \nu(y)=v\}\, dx\,dy.
\end{align*}
Let $\pi_{U=k}^{\bar{B}}, \pi_{V=k}^{\bar{D}}, \pi_{\mu=k}^{\bar{\beta}}$, and $\pi_{\nu=k}^{\bar{\delta}}$ be defined for $k\in [K]$ as
\begin{align*}
 \pi^{\bar{B}}_{U=k} & = \frac{1}{m} \sum_{i=1}^m \bar{B}_i\bar{B}^T_i 1\{U_i=k\} &  \pi^{\bar{D}}_{V=k} & = \frac{1}{n} \sum_{j=1}^n \bar{D}_j\bar{D}^T_j 1\{V_j =k\} \\
\pi^{\bar{\beta}}_{\mu=k} & = \int_{\Xscr} \bar{\beta}(x)\bar{\beta}(x)^T 1\{\mu(x)=k\}\, dx & \pi^{\bar{\delta}}_{\nu=k} & = \int_{\Yscr} \bar{\delta}(y)\bar{\delta}(y)^T 1\{\nu(y)=k\}\, dy. 
\end{align*}
We observe that $\sum_{k=1}^K \|\pi_{U=k}^{\bar{B}}\|_F \leq 1$, since by triangle inequality,
\begin{align*}
	\sum_{k=1}^K \|\pi_{U=k}^{\bar{B}}\|_F &\leq \frac{1}{m} \sum_{i=1}^m \|\bar{B}_i\bar{B}_i^T\|_F \\
	& \leq 1,
\end{align*}
where we have used $\|\bar{B}_i\| \leq 1$ for all $\bar{B}_i \in \Dscr$.

Recall the definitions for $g_y \in [0,1]^m$ and $f_x:\Yscr \mapsto [0,1]$:
\[ g_y(i) = \omega(x_i,y) \qquad \text{and} \qquad f_x(y) = \omega(x,y).\]
Define for $k \in [K]$ the matrices $g^{\bar{D}}_{V=k}$ and $g^{\bar{\delta}}_{\nu=k}$ in $[0,1]^{m \times d}$, and the functions $f^{\bar{B}}_{U=k}$ and $f^{\bar{\beta}}_{\mu=k}$ mapping $\Yscr \mapsto \Dscr$:
\begin{align*}
g^{\bar{D}}_{V=k} & = \frac{1}{n} \sum_{j=1}^n g_{y_j} \bar{D}_j^T 1\{V_j=k\}  &  g^{\bar{\delta}}_{\nu=k} & = \int_\Yscr g_y \bar{\delta}(y)^T 1\{\nu(y) = k\}\, dy \\
f^{\bar{B}}_{U=k} & = \frac{1}{m} \sum_{i=1}^m f_{x_i} \bar{B}_i 1\{U_i = k\} & f^{\bar{\beta}}_{\mu=k} & = \int_\Xscr f_x \bar{\beta}(x) 1\{\mu(x)=k\}\, dx.
\end{align*}
Define the matrix $1_{U=u}^{\bar{B}} \in [0,1]^{m \times d}$ and function $1_{\nu=v}^{\bar{\delta}}: \Yscr\mapsto \Dscr$
\begin{align*}
1_{U=u}^{\bar{B}}(i,j) & = \begin{cases} \bar{B}_i(j) & \text{if } U_i=u \\ 0 & \text{otherwise}\end{cases} &
1_{\nu=v}^{\bar{\delta}}(y) & = \begin{cases} \bar{\delta}(y) & \text{if } \nu(y) = v \\ 0 & \text{otherwise.} \end{cases}
\end{align*}
We observe that $m^{-1} \sum_{k=1}^K \|1_{U=k}^{\bar{B}}\|^2 \leq 1$ since $\|\bar{B}_i\|^2 \leq 1$ for all $\bar{B}_i \in \Dscr$. Analogous to $G_T, G_\tau, F_S, F_\sigma$ as defined in Section \ref{sec: proof sketch}, let $G_V^{\bar{D}}, G_\nu^{\bar{\delta}}, F_U^{\bar{B}},$ and $F_\mu^{\bar{\beta}}$ be defined by:
\begin{align*}
G_V^{\bar{D}} & = \left(\frac{g^{\bar{D}}_{V=1}}{\sqrt{m}}, \ldots, \frac{g^{\bar{D}}_{V=K}}{\sqrt{m}}, \pi^{\bar{D}}_{V=1}, \ldots, \pi^{\bar{D}}_{V=K}, \frac{\Psi_{\bar{T}}}{Kd} \right) & G_\nu^{\bar{\delta}} & = \left(\frac{g^{\bar{\delta}}_{\nu=1}}{\sqrt{m}},\ldots, \frac{g^{\bar{\delta}}_{\nu=K}}{\sqrt{m}}, \pi^{\bar{\delta}}_{\nu=1}, \ldots, \pi^{\bar{\delta}}_{\nu=K}, \frac{\Psi_{\bar{\tau}}}{Kd} \right) \\
F_U^{\bar{B}} & = \left(f^{\bar{B}}_{U=1}, \ldots, f^{\bar{B}}_{U=K}, \pi^{\bar{B}}_{U=1}, \ldots, \pi^{\bar{B}}_{U=K}, \frac{\Psi_{\bar{S}}}{Kd} \right) & F_\mu^{\bar{\beta}} & = \left(f^{\bar{\beta}}_{\mu=1},\ldots, f^{\bar{\beta}}_{\mu=K}, \pi^{\bar{\beta}}_{\mu=1}, \ldots, \pi^{\bar{\beta}}_{\mu=K}, \frac{\Psi_{\bar{\sigma}}}{Kd} \right).
\end{align*}
Define the sets $\bar{\Fscr}_m, \bar{\Gscr}_n, \bar{\Fscr},$ and $\bar{\Gscr}$ by
\begin{align*}
\bar{\Fscr}_n & = \{F_U^{\bar{B}}: \bar{S} = (U,\bar{B}) \in \bar{\Sscr}^m\} & \bar{\Fscr} & = \{F_{\mu}^{\bar{\beta}}: \bar{\sigma} = (\mu, \bar{\beta}) \in \Xscr \mapsto \bar{\Sscr}\} \\
\bar{\Gscr}_n & = \{G_V^{\bar{D}}: \bar{T} = (V,\bar{D}) \in \bar{\Tscr}^n\} & \bar{\Gscr} & = \{G_{\nu}^{\bar{\delta}}: \bar{\tau} = (\nu, \bar{\delta}) \in \Yscr \mapsto \bar{\Tscr}\} .
\end{align*}

\subsection{Intermediate Results for Proof of Theorem \ref{th: generalized}}


Lemmas \ref{le: th2 le2} and \ref{le: th2 le3} are analogs to Lemmas \ref{le: supporting hyperplanes} and \ref{le: convex}.

\begin{lemma}\label{le: th2 le2}
Under the conditions of Theorem 2, 
\begin{align}
\sup_{\|H\|=1} \left| \Gamma_{\bar{\Gscr}_n}(H) - \Gamma_{\bar{\Gscr}}(H)\right| & \leq O_P(\bar{K} (\log n) n^{-1/2}) \\
\sup_{\|H\|=1} \left| \Gamma_{\bar{\Fscr}_m}(H) - \Gamma_{\bar{\Fscr}}(H) \right| & \leq O_P(\bar{K} (\log m) m^{-1/2}), 
\end{align}
which implies
\begin{align*}
d_{\textrm{Haus}}(\operatorname{conv}(\bar{\Gscr}_n), \operatorname{conv}(\bar{\Gscr})) &\leq O_P(\bar{K}(\log n)n^{-1/2}) \\
d_{\textrm{Haus}}(\operatorname{conv}(\bar{\Fscr}_m), \operatorname{conv}(\bar{\Fscr})) &\leq O_P(\bar{K}(\log m)m^{-1/2}).
\end{align*}
\end{lemma} 

\begin{lemma}\label{le: th2 le3}
It holds that
\begin{align}
d_\textrm{Haus}(\operatorname{conv}(\bar{\Gscr}), \bar{\Gscr}) &= 0 \\
d_\textrm{Haus}(\operatorname{conv}(\bar{\Fscr}), \bar{\Fscr}) &= 0.
\end{align}
\end{lemma}

Lemmas \ref{le: th2 aux1} - \ref{le: th2 aux3} bound various error terms that appear in the proof of Theorem \ref{th: generalized}. They bound on the approximation error that arises when substituting $(\bar{S},\bar{T})$, and also the differences $|R_A(\bar{S}, \bar{T};\theta) - R_W(\bar{S}, \bar{T};\theta)|, |R_W(\bar{S}, \bar{T}; \theta) - R_\omega(\bar{S}, \bar{\tau};\theta)|$ and $|R_\omega(\bar{S}, \bar{\tau};\theta) - R_\omega(\bar{\sigma}, \bar{\tau}; \theta)|$.

\begin{lemma}\label{le: th2 aux1}
It holds that 
\begin{align} \label{eq: th2 proof1}
|R_A(S,T;\theta) - R_A(\bar{S}, \bar{T};\theta)| & \leq 12 \epsilon\\
\nonumber |R_\omega(S,\tau;\theta) - R_\omega(\bar{S}, \tau;\theta)| & \leq 12\epsilon\\
\nonumber |R_\omega(\sigma, \tau;\theta) - R_\omega(\sigma, \bar{\tau}; \theta)| &\leq 12 \epsilon.
\end{align}
and that
\begin{align} \label{eq: th2 proof1a} 
\|\Psi_S - \Psi_{\bar{S}}\|^2 & \leq Kd\epsilon & 
\|\Psi_T - \Psi_{\bar{T}}\|^2 & \leq Kd\epsilon \\
\nonumber \|\Psi_\sigma - \Psi_{\bar{\sigma}}\|^2 & \leq Kd\epsilon & 
\|\Psi_\tau - \Psi_{\bar{\tau}}\|^2 & \leq Kd\epsilon.
\end{align}
\end{lemma}

\begin{lemma}\label{le: th2 aux2}
If $\bar{K} \leq n^{1/2}$, it holds that
\begin{align}
|R_A(\bar{S}, \bar{T};\theta) - R_W(\bar{S}, \bar{T}; \theta) - C_1| & \leq 2 \bar{K} O_P(\bar{K}(\log n)(n^{-1}) \label{eq: th2 proof2},
\end{align}
\end{lemma}

\begin{lemma}\label{le: th2 aux3}
Given $\bar{T} = (V,\bar{D}) \in \bar{\Tscr}^n$, let $\bar{\tau} = (\nu, \bar{\delta}) \in \Yscr \mapsto \bar{\Tscr}$ minimize $\|G_V^{\bar{D}} - G_\nu^{\bar{\delta}}\|$. It holds that
\begin{align}
|R_W(\bar{S}, \bar{T};\theta) - R_\omega(\bar{S}, \bar{\tau};\theta) - C_2| & \leq O_P(K\bar{K} (\log n)n^{-1/2}). \label{eq: th2 proof3} 
\end{align}
Given $\bar{S} = (U,\bar{B}) \in \bar{\Sscr}^n$, let $\bar{\sigma} = (\mu, \bar{\beta}) \in \Xscr \mapsto \bar{\Sscr}$ minimize $\|F_U^{\bar{B}} - F_\mu^{\bar{\beta}}\|$. It holds that
\begin{align}
|R_\omega(\bar{S}, \bar{\tau};\theta) - R_\omega(\bar{\sigma}, \bar{\tau}; \theta) - C_3|  & \leq O_P(K \bar{K}(\log m)m^{-1/2}). \label{eq: th2 proof4}
\end{align}
\end{lemma}

\subsection{Proof of Theorem \ref{th: generalized}}

\begin{proof}[Proof of Theorem \ref{th: generalized}]

Given $\bar{T} = (V,\bar{D})$, let $\bar{\tau} = (\nu, \bar{\delta})$ minimize $\|G_V^{\bar{D}} - G_\nu^{\bar{\delta}}\|$, which by Lemmas \ref{le: th2 le2} and \ref{le: th2 le3} is bounded by $O_P(\bar{K}(\log n) n^{-1/2})$. Using this fact and \eqref{eq: th2 proof1a}, the quantity $\|\Psi_T - \Psi_{\tau}\|^2$ can be bounded by
\begin{align} 
\nonumber	\|\Psi_T - \Psi_{\tau}\|^2 & \leq 2\|\Psi_T - \Psi_{\bar{T}}\|^2 + 2\|\Psi_{\bar{T}} - \Psi_{\bar{\tau}}\|^2 \\
\nonumber	& \leq 2\|\Psi_T - \Psi_{\bar{T}}\|^2 + 2\|G_V^{\bar{D}} - G_\nu^{\bar{\delta}}\|^2 \\
	& \leq 2Kd\epsilon + O_P(\bar{K}^2(\log n)n^{-1}),\label{eq: Psi_T}.
\end{align}
Using \eqref{eq: th2 proof1}, \eqref{eq: th2 proof2}, \eqref{eq: th2 proof3}, \eqref{eq: th2 proof4}, and \eqref{eq: Psi_T}, it holds for $\bar{K} \leq n^{1/2}$ that 
\begin{align} 
\nonumber |R_A(S,T;\theta) - R_\omega(S,\bar{\tau};\theta) - C_1 - C_2| + \frac{\|\Psi_T - \Psi_{\bar{\tau}}\|^2}{Kd} & \leq |R_A(S,T;\theta) - R_A(\bar{S}, \bar{T};\theta)|\\
\nonumber & \qquad {} + |R_A(\bar{S}, \bar{T};\theta) - R_W(\bar{S}, \bar{T}; \theta) - C_1 | \\
\nonumber & \qquad {} + |R_W(\bar{S}, \bar{T}; \theta) - R_\omega(\bar{S}, \bar{\tau};\theta) - C_2| \\
\nonumber & \qquad {} + |R_\omega(\bar{S}, \bar{\tau};\theta) - R_\omega(S, \bar{\tau};\theta)| \\
\nonumber & \qquad {} + \frac{\|\Psi_T - \Psi_{\tau}\|^2}{Kd} \\
& \leq 26\epsilon + O_P\left(\bar{K}^2 \frac{\log(n)}{n}\right) + O_P\left(K\bar{K} \frac{\log(n)}{n^{1/2}}\right). \label{eq: th2 proof5}
\end{align}
Using $\bar{K} \leq K(d^{1/2} \epsilon^{-1})^d$ and letting $\epsilon = \left(\frac{K^2 d^{d/2} \log n}{n^{1/2}}\right)^{\frac{1}{1+d}}$ yields that $\bar{K} \leq n^{1/2}$ eventually, so that substituting into \eqref{eq: th2 proof5} yields
\[ |R_A(S,T;\theta) - R_\omega(S,\bar{\tau};\theta) - C_1 - C_2| + \frac{\|\Psi_T - \Psi_{\bar{T}}\|^2}{Kd} \leq O_P\left(d^{1/2} \left(\frac{K^2 \log n}{n^{1/2}}\right)^{\frac{1}{1+d}}\right),\]
proving \eqref{eq: R_tau}.

Similarly, it holds that
\begin{align*}
|R_\omega(S,\tau;\theta) - R_\omega(\bar{\sigma}, \tau;\theta) - C_3| + \frac{\|\Psi_S - \Psi_{\bar{\sigma}}\|^2}{Kd} &\leq |R_\omega(S,\tau;\theta) - R_\omega(\bar{S}, \bar{\tau};\theta)| \\
& \qquad {} + |R_\omega(\bar{S}, \bar{\tau};\theta) - R_\omega(\bar{\sigma}, \bar{\tau};\theta) - C_3| \\
& \qquad {} + |R_\omega(\bar{\sigma}, \bar{\tau};\theta) - R_\omega(\bar{\sigma}, \tau;\theta)|  \\
& \qquad {} + \frac{\|\Psi_S - \Psi_{\bar{\sigma}}\|}{Kd} \\
& \leq 26 \epsilon + O_P\left(\bar{K}^2 \frac{\log(m)}{m}\right) + O_P\left(K\bar{K} \frac{\log(m)}{m^{1/2}}\right),
\end{align*}
and letting $\epsilon = \left(\frac{K^2 d^{d/2} \log m}{m^{1/2}}\right)^{\frac{1}{1+d}}$ proves \eqref{eq: R_sigma}.
\end{proof}

\subsection{Proof of Lemmas \ref{le: th2 le2} - \ref{le: th2 aux3}}

\begin{proof}[Proof of Lemma \ref{le: th2 le2}]
Let  $H = (h_1, \ldots, h_K, \pi_1, \ldots, \pi_v, \Psi_H)$, where $h_k \in \mathbb{R}^{m \times d}, \pi_k \in \mathbb{R}^{d\times d}$, and $\Psi_H:[K]\times \Dscr \mapsto [0,1]$. Given $(v,d) \in [K] \times \Dscr$, let $1_{v,d}: [K] \times \Dscr \mapsto [0,1]$ denote the indicator function
\[ 1_{v,d}(k,c) = 1\{v \leq k, d \leq c\}.\]
Given $G_V^{\bar{D}} \in \bar{\Gscr}_n$ and $G_\nu^{\bar{\delta}} \in \bar{\Gscr}$, the inner products $\ip{H, G_V^{\bar{D}}}$ and $\ip{H, G_\nu^{\bar{\delta}}}$ equal
\begin{align*}
\ip{H, G_V^{\bar{D}}} & = \sum_{k=1}^K \ip{h_k, \frac{g_{V=k}^{\bar{D}}}{\sqrt{m}}} + \sum_{k=1}^K \ip{\pi_k, \pi_{V=k}^{\bar{D}}} + \frac{1}{Kd}\ip{\Psi_H, \Psi_{\bar{T}}}\\
& = \frac{1}{n} \sum_{j=1}^n \left[ \ip{h_{V_j}, \frac{g_{y_j}\bar{D}_j^T }{\sqrt{m}}} + \ip{\pi_{V_j}, \bar{D}_j \bar{D}_j^T} + \frac{1}{Kd}\ip{\Psi_H, 1_{V_j,\bar{D}_j}}\right] \\
\ip{H, G_\nu^{\bar{\delta}}} & = \sum_{k=1}^K \ip{h_k, \frac{g_{\nu=k}^{\bar{\delta}}}{\sqrt{m}}} + \sum_{k=1}^K \ip{\pi_k, \pi_{\nu=k}^{\bar{\delta}}} + \frac{1}{Kd}\ip{\Psi_H, \Psi_{\bar{\tau}}}\\
& = \int_{\Yscr} \ip{h_{\nu(y)}, \frac{ g_y\bar{\delta}(y)^T}{\sqrt{m}}} + \ip{\pi_{\nu(y)}, \bar{\delta}(y)\bar{\delta}(y)^T} + \frac{1}{Kd}\ip{\Psi_H, 1_{\nu(y), \bar{\delta}(y)}}\, dy.
\end{align*}
It follows that the support functions $\Gamma_{\bar{\Gscr}_n}$ and $\Gamma_{\bar{\Gscr}}$ equal
\begin{align*}
\Gamma_{\bar{\Gscr}_n}(H) & = \frac{1}{n} \sum_{j=1}^n \max_{v \in [K], c \in \bar{\Dscr}} \left[ \ip{h_v , \frac{g_{y_j}c^T}{\sqrt{m}}} + \ip{\pi_{v}, cc^T} + \frac{1}{Kd}\ip{\Psi_H, 1_{v,c}}\right] \\
& = \frac{1}{n} \sum_{j=1}^n \max_{(v,c) \in \bar{\Tscr}} \ip{\left[\begin{array}{c} h_vc \\ \ip{\pi_v,cc^T} \\ (Kd)^{-1}\ip{\Psi_H, 1_{v,c}}\end{array}\right], \left[\begin{array}{c} m^{-1/2} g_{y_j} \\ 1 \\ 1\end{array}\right]} \\
\Gamma_{\bar{\Gscr}}(H) &= \int_\Yscr \max_{v \in [K], c \in \bar{\Dscr}} \left[ \ip{h_v , \frac{g_{y}c^T}{\sqrt{m}}} + \ip{\pi_{v}, cc^T} + \frac{1}{Kd}\ip{\Psi_H, 1_{v,c}}\right] \\
& = \int_\Yscr \max_{(v,c) \in \bar{\Tscr}} \ip{\left[\begin{array}{c} h_vc \\ \ip{\pi_v,cc^T} \\ (Kd)^{-1}\ip{\Psi_H, 1_{v,c}}\end{array}\right], 
\left[\begin{array}{c} m^{-1/2} g_{y} \\ 1 \\ 1 \end{array}\right]}\, dy.
\end{align*}
Given $t = (v,c) \in \bar{\Tscr}$, let $h_t' = \left[\begin{array}{c} h_v c \\ \ip{\pi_v, cc^T} \\ (Kd)^{-1}\ip{\Psi_H, 1_{v,c}}\end{array}\right]$. Since $\|c\|\leq 1$ and $\| (Kd)^{-1}1_{v,c}\| \leq 1/d$, it follows that $\|h_t'\|^2 \leq \|h_v\|_F^2 + \|\pi_v\|_F^2 + \|\Psi_H\|^2/d^2$, where $\|\cdot\|_F$ denotes Frobenius norm, so that if $\|H\|\leq 1$, then $\|h_t'\| \leq 1$ for all $t \in \bar{\Tscr}$. As a result, the proof of Lemma \ref{le: supporting hyperplanes} can be copied here: Lemma \ref{le: Biau} implies that 
\[ \mathbb{E}\sup_{\|H\|=1} |\Gamma_{\bar{\Gscr}}(H) - \Gamma_{\bar{\Gscr}_n}(H)| \leq \frac{6\bar{K}}{\sqrt{n}},\]
and McDiarmid's inequality applied to $Z = \sup_{\|H\|=1} |\Gamma_{\bar{\Gscr}_n}(H) - \Gamma_{\bar{\Gscr}}(H)|$ implies that $Z - \mathbb{E}Z = O_P(n^{-1/2}\log n)$.

The proof for $\sup_{\|H\|=1} |\Gamma_{\bar{\Fscr}_m}(H) - \Gamma_{\bar{\Fscr}}(H)|$ follow parallel arguments.
\end{proof}.

\begin{proof}[Proof of Lemma \ref{le: th2 le3}]
Enumerate the members of $\bar{\Tscr}$ as $1,\ldots,\bar{K}$. Given $(u, c) \in \bar{\Tscr}$, let $t(u,c)$ denote its corresponding index in $1,\ldots,\bar{K}$. Given $\bar{T} = (V,\bar{D}) \in \bar{\Tscr}$, recall the definition of 
\[G_{\bar{T}} = \left( \frac{g_{{\bar{T}=1}}}{\sqrt{m}}, \ldots, \frac{g_{{\bar{T}=\bar{K}}}}{\sqrt{m}}, \pi_{{\bar{T}}}  \right),\]
a vector in $\mathbb{R}^{m\bar{K} + \bar{K}}$.  It can be seen that 
\[G_V^{\bar{D}} = \left( \frac{g_{V=1}^{{\bar{D}}}}{\sqrt{m}}, \ldots, \frac{g_{V=K}^{\bar{D}}}{\sqrt{m}}, \pi_{V=1}^{\bar{D}}, \ldots, \pi_{V=K}^{\bar{D}}, \frac{\Psi_{\bar{T}}}{Kd}\right)\] is a linear transformation of $G_{\bar{T}}$, given by
\begin{align*}
g_{V=k}^{\bar{D}} &= \sum_{c \in \bar{\Dscr}} g_{{\bar{T}}=t(k,c)}c^T, \qquad k \in [K] \\
\pi_{V=k}^{\bar{D}} &= \sum_{c \in \bar{\Dscr}} \pi_{\bar{T}}(t(k,c)) cc^T, \qquad k \in [K] \\
\Psi_{\bar{T}} &= \sum_{k=1}^K \sum_{c \in \bar{\Dscr}} \pi_{t(k,c)} 1_{k,c}.
\end{align*}
By Lemma \ref{le: convex}, it holds that $\Gscr = \{G_{\bar{T}}: \bar{T} \in \bar{\Tscr}\}$ is convex. Since $\bar{\Gscr} = \{G_V^{\bar{D}}: \bar{T}= (V,\bar{D}) \in \bar{\Tscr}\}$ is related to $\Gscr$ by a linear transform, and as linear transformations preserve convexity, it follows that $\bar{\Gscr}$ is also convex. By parallel arguments, it also follows that $\bar{\Fscr}$ is a linear transformation of $\Fscr$ and hence convex as well.
\end{proof}

\begin{proof}[Proof of Lemma \ref{le: th2 aux1}]
If $\|B_i - \bar{B_i}\| \leq \epsilon$ for all $i\in[m]$ and $\|D_i - \bar{D}_i\| \leq \epsilon$ for all $j \in [n]$, then
\begin{align*}
|R_A(S,T;\theta) - R_A(\bar{S}, \bar{T};\theta)| & \leq \left|\frac{1}{mn} \sum_{i=1}^m \sum_{j=1}^n (A_{ij} - B_i^TD_j \theta_{U_i V_j})^2 - (A_{ij} - \bar{B}_i^T \bar{D}_j \theta_{U_i V_j})^2 \right| \\
 &\leq 12\epsilon,
\end{align*}
where we use the fact that $\|B_i\|, \|D_j\|, \theta_{uv}$, and $A_{ij}$ are all between $0$ and $1$. By similar arguments, it also holds that
\begin{align*}
	|R_\omega(S,\tau;\theta) - R_\omega(\bar{S}, \tau;\theta)| & \leq 12\epsilon\\
	|R_\omega(\sigma, \tau;\theta) - R_\omega(\sigma, \bar{\tau}; \theta)| &\leq 12 \epsilon.
\end{align*}
We also show that $\|\Psi_S - \Psi_{\bar{S}}\|^2 \leq Kd\epsilon$ by 
\begin{align*}
	\| \Psi_S - \Psi_{\bar{S}}\|^2 & = \sum_{k=1}^K \int_{\lbrack 0,1)^d} \left[ \frac{1}{m} \sum_{i=1}^m \big( 1\{U_i \leq k, \beta_i \leq c\} - 1\{U_i \leq k, \bar{\beta}_i \leq c\}\big) \right]^2\, dc \\
	& \leq \sum_{k=1}^K \int_{ \lbrack 0,1)^d} \frac{1}{m} \sum_{i=1}^m \big[ 1\{U_i \leq k, \beta_i \leq c\} - 1\{U_i \leq k, \bar{\beta}_i \leq c\}\big]^2\, dc \\
	& = \frac{1}{m} \sum_{i=1}^m \sum_{k=1}^K \int_{ \lbrack 0,1)^d} \big| 1\{U_i \leq k, \beta_i \leq c\} - 1\{U_i \leq k,\bar{\beta}_i \leq c\} \big|\, dc \\
	& \leq Kd\epsilon, 
\end{align*}
where the first inequality holds by Jensen's inequality, and the second inequality holds because $\|\beta_i - \bar{\beta}_i\| \leq \epsilon$ and the integral is over $[0,1)^d$. The quantities $\|\Psi_T - \Psi_{\bar{T}}\|^2, \|\Psi_\sigma - \Psi_{\bar{\sigma}}\|^2$, etc., are bounded similarly.
\end{proof}

\begin{proof}[Proof of Lemma \ref{le: th2 aux2}]
Given $\theta \in [0,1]^{K \times K}$, let $\bar{\theta} \in [0,1]^{\bar{K} \times \bar{K}}$ be given by
\[\bar{\theta}_{st} = \bar{b}^T\bar{d}\, \theta_{uv} \qquad \text{for all } s = (u,\bar{b}) \in \bar{\Sscr}, t = (v,\bar{d}) \in \bar{\Tscr}.\]
For $\bar{S} = (U, \bar{B}) \in \bar{S}^m$ and $\bar{T} = (V, \bar{D}) \in \bar{T}^n$, 
\begin{align} \label{eq: th2 diff2}
R_A(\bar{S}, \bar{T};\theta) - R_W(\bar{S}, \bar{T}; \theta) = C_1 - 2\sum_{s \in \bar{\Sscr}} \sum_{t \in \bar{\Tscr}} ([\Phi_A(\bar{S},\bar{T})]_{st} - [\Phi_W(\bar{S}, \bar{T})]_{st}) \bar{\theta}_{st},
\end{align}
where $C_1$ is constant in $(\bar{S},\bar{T},\theta)$. This implies
\begin{align*}
 |R_A(\bar{S}, \bar{T};\theta) - R_W(\bar{S}, \bar{T}; \theta) - C_1| &\leq 2\|\Phi_A(\bar{S},\bar{T}) - \Phi_W(\bar{S}, \bar{T})\|_1 \\
& \leq 2 \bar{K}\|\Phi_A(\bar{S},\bar{T}) - \Phi_W(\bar{S}, \bar{T})\|_2 \\
& \leq 2 \bar{K} O_P(\bar{K}(\log n)n^{-1})
\end{align*}
where the inequalities follow by \eqref{eq: th2 diff2}, the equivalence of norms, and Lemma \ref{le: mcdiarmid}, which requires  $\bar{K} \leq n^{1/2}$. 
\end{proof}

\begin{proof}[Proof of Lemma \ref{le: th2 aux3}]
It holds that
\begin{align*}
\nonumber R_W(\bar{S}, \bar{T}; \theta) - R_\omega(\bar{S}, \bar{\tau};\theta) & = C_2 - 2\sum_{u=1}^K \sum_{v=1}^K \left( [\Phi_Z(U,V)]_{uv} - [\Phi_\zeta(U,V)]_{uv}\right) \theta_{uv} \\
& \qquad { } + \sum_{u=1}^K \sum_{v=1}^K \left( \ip{\pi_{U=u}^{\bar{B}}, \pi_{V=v}^{\bar{b}}} - \ip{\pi_{U=u}^{\bar{B}}, \pi_{\nu=v}^{\bar{\delta}}}\right)\theta_{uv}^2. 
\end{align*}
where $C_2$ is constant in $\bar{S},\bar{T},\theta$ and $\bar{\tau}$. This implies
\begin{align}
\nonumber |R_W(\bar{S}, \bar{T};\theta) - R_\omega(\bar{S}, \bar{\tau};\theta) - C_2| & \leq \|\Phi_Z(U,V) - \Phi_\zeta(U, \nu)\|_1 \\
\nonumber & \qquad { }+ \left(\sum_{u=1}^K \|\pi_{U=u}^{\bar{B}}\| \right) \left(\sum_{v=1}^K \|\pi_{V=v}^{\bar{b}} - \pi_{\nu=v}^{\bar{\delta}}\|\right) \\
& \leq K \|\Phi_Z(U,V) - \Phi_\zeta(U, \nu)\| + \sqrt{K} \left(\sum_{v=1}^K \| \pi_{V=v}^{\bar{b}} - \pi_{\nu=v}^{\bar{b}}\|^2\right)^{1/2}, \label{eq: th2 diff3}
\end{align}
where the final inequality uses the fact that $\sum_{u=1}^K \|\pi_{U=u}^{\bar{B}}\| \leq 1$.

It can be seen that the entries of $\Phi_Z(U,V)$ and $\Phi_\zeta(U, \nu)$ equal the inner products
\begin{align*}
[\Phi_Z(U,V)]_{uv} &= \frac{1}{m} \ip{1_{U=u}^{\bar{B}}, g_{V=v}^{\bar{D}}}  & [\Phi_\zeta(U,\nu)]_{uv} &= \frac{1}{m} \ip{1_{U=u}^{\bar{B}}, g_{\nu = v}^{\bar{\delta}}}, 
\end{align*}
which implies
\begin{align} 
\nonumber \|\Phi_Z(U,V) - \Phi_{\zeta}(U, \nu)\|^2 & \leq \left(\sum_{u=1}^K \frac{1}{m} \|1_{U=u}^{\bar{B}} \|^2 \right) \left(\sum_{v=1}^K \frac{1}{m}\|g_{V=v}^{\bar{D}} - g_{\nu=v}^{\bar{\delta}}\|^2\right) \\
& \leq \sum_{v=1}^K \frac{1}{m}\|g_{V=v}^{\bar{D}} - g_{\nu=v}^{\bar{\delta}}\|^2, \label{eq: th2 ip}
\end{align}
Given $G_V^{\bar{D}} \in \bar{\Gscr}_n$, let $G_{\nu}^{\bar{\delta}} \in \bar{\Gscr}$ minimize $\|G_V^{\bar{D}} - G_{\nu}^{\bar{\delta}}\|$. Using \eqref{eq: th2 diff3}, \eqref{eq: th2 ip} and Lemma \ref{le: th2 le2} implies 
\begin{align*}
|R_W(\bar{S}, \bar{T};\theta) - R_\omega(\bar{S}, \bar{\tau};\theta) - C_2| & \leq \sqrt{2}K \|G_V^{\bar{D}} - G_{\nu}^{\bar{\delta}}\| \\
& \leq O_P(K\bar{K} (\log n)n^{-1/2}) 
\end{align*}

Similarly, it holds that 
\begin{align*}
\nonumber R_\omega(\bar{S}, \bar{\tau};\theta) - R_\omega(\bar{\sigma}, \bar{\tau}; \theta) & = C_3 - 2 \sum_{u=1}^K \sum_{v=1}^K \left( [\Phi_\zeta(U, \nu)]_{uv} - [\Phi_\zeta(\mu, \nu)]_{uv}\right)\theta_{uv} \\
& \qquad { } + \sum_{u=1}^K \sum_{v=1}^K \left(\ip{\pi_{U=u}^{\bar{B}}, \pi_{\nu=v}^{\bar{\delta}}} - \ip{\pi_{\mu=u}^{\bar{\beta}}, \pi_{\nu=v}^{\bar{\delta}}}\right)\theta_{uv}^2, 
\end{align*}
where $C_3$ is constant in $\bar{S}, \bar{\tau}, \bar{\sigma}$ and $\theta$. This implies
\begin{align}
\nonumber |R_\omega(\bar{S}, \bar{\tau};\theta) - R_\omega(\bar{\sigma}, \bar{\tau}; \theta) - C_3| &\leq 2 \|\Phi_\zeta(U,\nu) - \Phi_\zeta(\mu, \nu)\|_1 \\
\nonumber & \qquad {} + \left(\sum_{u=1}^K \|\pi_{U=u}^{\bar{B}} - \pi_{\mu=u}^{\bar{\beta}}\|\right) \left(\sum_{v=1}^K \| \pi_{\nu=v}^{\bar{\delta}}\|\right) \\
&\leq 2K \|\Phi_\zeta(U,\nu) - \Phi_\zeta(\mu, \nu)\| + \sqrt{K} \left(\sum_{u=1}^K \|\pi_{U=u}^{\bar{B}} - \pi_{\mu=u}^{\bar{\beta}}\|^2\right)^{1/2}, \label{eq: th2 diff4}
\end{align}
where we have used $\sum_v \|\pi_{\nu=v}^{\bar{\delta}}\| \leq 1$. 

The entries of $\Phi_\zeta(U, \nu)$, and $\Phi_\zeta(\mu,\nu)$ equal the inner products
\[
[\Phi_\zeta(U,\nu)]_{uv} = \ip{f_{U=u}^{\bar{B}}, 1_{\nu=v}^{\bar{\delta}}} \qquad \text{and} \qquad [\Phi_\zeta(\mu, \nu)]_{uv} = \ip{f_{\mu=u}^{\bar{\beta}}, 1_{\nu=v}^{\bar{\delta}}},
\]
which implies 
\begin{align}
\nonumber \|\Phi_\zeta(U,\nu) - \Phi_\zeta(\mu,\nu)\|^2 & \leq \left(\sum_{u=1}^K \frac{1}{m} \|f_{U=u}^{\bar{B}} - f_{\mu=u}^{\bar{\beta}}\|^2 \right) \left(\sum_{v=1}^K \frac{1}{m}\| 1_{\nu=v}^{\bar{\delta}} \|^2\right) \\
& \leq \sum_{u=1}^K \frac{1}{m} \|f_{U=u}^{\bar{B}} - f_{\mu=u}^{\bar{\beta}}\|^2. \label{eq: th2 ip2}
\end{align}
Given $F_U^{\bar{B}} \in \bar{\Fscr}_m$, let $F_{\mu}^{\bar{\beta}} \in \bar{\Fscr}$ minimize $\|F_U^{\bar{B}} - F_{\mu}^{\bar{\beta}}\|$. It follows from \eqref{eq: th2 diff4}, \eqref{eq: th2 ip2}, and Lemma \ref{le: th2 le2} that 
\begin{align*}
\nonumber |R_\omega(\bar{S}, \bar{\tau};\theta) - R_\omega(\bar{\sigma}, \bar{\tau}; \theta) - C_3|  & \leq \sqrt{2}K \|F_U^{\bar{B}} - F_{\mu}^{\bar{\beta}}\| \\
& \leq O_P(K \bar{K}(\log m)m^{-1/2}).
\end{align*}
\end{proof}

\bibliographystyle{plain}
\bibliography{bibfile}
\end{document}